\documentclass[11pt]{scrartcl}
\usepackage{amsmath,amsthm,amscd}
\usepackage[mdbch]{mathdesign}
\usepackage{hyperref}
\usepackage{enumerate}
\usepackage{multicol}
\usepackage{color}
\usepackage{a4wide}

\usepackage{tikz-cd}
\usetikzlibrary{arrows, matrix}
\usepackage{tikz}
\usetikzlibrary{decorations.markings}
\usetikzlibrary{calc}
\tikzset{
directed/.style={postaction={decorate,decoration={
  markings,
  mark=at position .5 with {\arrow{>}}  
  }}},
vertex/.style={draw, circle, fill=white, minimum size=5pt, inner sep=0pt},
edge/.style={thick, directed},
leaf/.style={thick, line cap=round, directed},
disconnected leaf/.style={thick, line cap=round},
}

\numberwithin{equation}{section}
\theoremstyle{plain}
\newtheorem{thm}{Theorem}[section]
\newtheorem*{thm*}{Theorem}
\newtheorem{prp}[thm]{Proposition}
\newtheorem{cor}[thm]{Corollary}
\newtheorem{lem}[thm]{Lemma}

\theoremstyle{definition}
\newtheorem{dfn}[thm]{Definition}
\newtheorem{example}[thm]{Example}
\newtheorem{rem}[thm]{Remark}

\title{Twisted generalized Weyl algebras and primitive quotients of enveloping algebras}
\author{Jonas T. Hartwig\and Vera Serganova}
\date{ }

\newcommand\al{\alpha} \newcommand\ga{\gamma}
\newcommand\ep{\varepsilon}
 \newcommand\la{\lambda}
  \newcommand\si{\sigma}

\newcommand{\Ga}{\Gamma}

    \newcommand\Z{\mathbb{Z}}   \newcommand\Q{\mathbb{Q}}
\newcommand\K{\Bbbk}    
\newcommand\Supp{\mathrm{Supp}}

\newcommand\Ann{\mathrm{Ann}}

\DeclareMathOperator{\End}{End}\DeclareMathOperator{\Hom}{Hom}
\DeclareMathOperator{\Aut}{Aut}\DeclareMathOperator{\Id}{Id}
\DeclareMathOperator{\ad}{ad}

\newcommand{\iv}[2]{\llbracket #1,#2 \rrbracket}

\newcommand{\TGWC}[4]{\mathcal{C}_{#1}({#2},{#3},{#4})}  
\newcommand{\TGWA}[4]{\mathcal{A}_{#1}({#2},{#3},{#4})}
\newcommand{\TGWI}[4]{\mathcal{I}_{#1}({#2},{#3},{#4})}

\renewcommand{\hat}{\widehat}

\begin{document}
\maketitle
\begin{abstract}

To each multiquiver $\Gamma$ we attach a solution to the consistency equations associated to twisted generalized Weyl (TGW) algebras. This generalizes several previously obtained solutions in the literature. We show that the corresponding algebras $\mathcal{A}(\Gamma)$ carry a canonical representation by differential operators and that $\mathcal{A}(\Gamma)$ is universal among all TGW algebras with such a representation. We also find explicit conditions in terms of $\Ga$ for when this representation is faithful or locally surjective. 

By forgetting some of the structure of $\Gamma$ one obtains a Dynkin diagram, $D(\Gamma)$.
We show that the generalized Cartan matrix of $\mathcal{A}(\Gamma)$ coincides with the one corresponding to $D(\Gamma)$ and that $\mathcal{A}(\Gamma)$ contains graded homomorphic images of the enveloping algebra of the positive and negative part of the corresponding Kac-Moody algebra.

Finally, we show that a primitive quotient $U/J$ of the enveloping algebra of a finite-dimensional simple Lie algebra over an algebraically closed field of characteristic zero is graded isomorphic to a TGW algebra if and only if $J$ is the annihilator of a completely pointed (multiplicity-free) simple weight module. The infinite-dimensional primitive quotients in types $A$ and $C$ are closely related to $\mathcal{A}(\Gamma)$ for specific $\Gamma$.
We also prove one result in the affine case.
\end{abstract}

\tableofcontents

\section{Preliminaries}
\subsection{Introduction}

Twisted generalized Weyl (TGW) algebras were introduced by Mazorchuk and Turowska in 1999 \cite{MazTur1999}.
They are constructed from a commutative unital ring $\K$, an associative unital $\K$-algebra $R$, a collection of elements $t=\{t_i\}_{i\in I}$ from the center of $R$, a collection of commuting $\K$-automorphisms $\si=\{\si_i\}_{i\in I}$ of $R$, and a scalar matrix $\mu=(\mu_{ij})_{i,j\in I}$, by adjoining to $R$ new non-commuting generators $X_i$ and $Y_i$ for $i\in I$, imposing some relations and taking the quotient by a certain radical $\mathcal{I}$ (see Section \ref{sec:background} for the full definition). These algebras are denoted by $\mathcal{A}_\mu(R,\si,t)$.
They are naturally graded by the free abelian group on the index set $I$ and come with a canonical map $R\to\mathcal{A}_\mu(R,\si,t)$ making them $R$-rings. In addition, if $\mu$ is symmetric, they can be equipped with an involution.

The structure and representation theory of TGW algebras have been investigated in several papers.
For example, families of simple weight modules were classified in \cite{MazTur1999},\cite{MazPonTur2003},\cite{Hartwig2006}, Whittaker modules classified in \cite{FutHar2012a},
bounded and unbounded $\ast$-representations studied in \cite{MazTur2002},
generalized Serre relations were found in \cite{Hartwig2010}, and 
conditions for a TGW algebra to be a simple ring were given in \cite{HarOin2013}.

Examples of TGW algebras include multiparameter quantized Weyl algebras \cite{MazTur2002}, \cite{Hartwig2006}, \cite{FutHar2012a}, $U(\mathfrak{sl}_2)$, $U_q(\mathfrak{sl_2})$, $Q_{ij}$-CCR (Canonical Commutation Relation) algebras \cite{MazTur2002}, quantized Heisenberg algebras, extended OGZ algebras $\mathcal{U}(r-1,r,r+1)$ \cite{MazPonTur2003},  the Mickelsson-Zhelobenko algebra associated to the pair $(\mathfrak{gl}_n,\mathfrak{gl}_{n-1})$ \cite{MazPonTur2003}, an example related to $\mathfrak{gl}_3$ \cite{Sergeev2001}, and examples attached to any symmetric generalized Cartan matrix \cite{Hartwig2010}.

In addition, any higher rank generalized Weyl algebra (GWA) \cite{Bavula1992} in which $X_i$ and $Y_i$ are not zero-divisors is also an example of a TGW algebra. The GWAs are obtained precisely when the following additional condition  holds:
\begin{equation}\label{eq:GWA-consistency}
\si_i(t_j)=t_j,\qquad i\neq j.
\end{equation}
In a GWA one has $X_iX_j=X_jX_i$ and $Y_iY_j=Y_jY_i$ for all $i,j$. These relations need not hold in a general TGW algebra, where instead they are replaced by higher degree Serre-type relations \cite{Hartwig2010}.

A question of particular importance is whether a given input datum $(R,\si,t)$ actually gives rise to a non-trivial TGW algebra. Indeed, it can happen that the relations are contradictory so that $\mathcal{A}_\mu(R,\si,t)=\{0\}$, the algebra with one element \cite[Ex.~2.8]{FutHar2012b}.
This does not occur for higher rank GWAs, as the conditions \eqref{eq:GWA-consistency} implies the algebra is consistent. Thus it becomes important to find a substitute for \eqref{eq:GWA-consistency} which ensures that a given TGW algebra is non-trivial. This question was solved in \cite{FutHar2012b}, in the case when the $t_i$ are not zero-divisors in $R$. The answer is the following.
\begin{thm*}[\cite{FutHar2012b}] Assume that the elements $t_i$ are not zero-divisors in $R$. Then the following set of equations is sufficient for a TGW algebra $\mathcal{A}_{\mu}(R,\si,t)$ to be non-trivial:
\begin{subequations}\label{eq:intro-tgw-consistency}
\begin{align}
\label{eq:intro-tgw-consistency1}
\si_i\si_j(t_it_j) &=\mu_{ij}\mu_{ji}\si_i(t_i)\si_j(t_j),\qquad i\neq j \\ 
\label{eq:intro-tgw-consistency2}
\si_i\si_k(t_j)t_j &=\si_i(t_j)\si_k(t_j),\qquad i\neq j\neq k\neq i
\end{align}
\end{subequations}
Moreover, if one requires the canonical map $R\to\mathcal{A}_\mu(R,\si,t)$ to be injective, then equations \eqref{eq:intro-tgw-consistency} are necessary.
\end{thm*}
Clearly, any solution to \eqref{eq:GWA-consistency} is also a solution to \eqref{eq:intro-tgw-consistency}, if we take $\mu_{ij}=1$ for all $i,j$. The first equation \eqref{eq:intro-tgw-consistency1} was known already in \cite{MazTur1999},\cite{MazTur2002}. The second equation \eqref{eq:intro-tgw-consistency2} was found in \cite{FutHar2012b} and is independent of the first one.
Hence, constructing examples of TGW algebras is equivalent to finding solutions $(R,\si,t)$ to the consistency equations \eqref{eq:intro-tgw-consistency}. This task is much more difficult than solving the GWA equation \eqref{eq:GWA-consistency}.

In the present paper we construct a new family of solutions to \eqref{eq:intro-tgw-consistency}, parametrized by multiquivers. A multiquiver is a quiver (directed graph) where we allow multiple edges between two vertices.

The motivation for this construction came from several directions. Already in \cite{MazTur1999}, the authors essentially attached a graph to a TGW algebra such that vertices $i$ and $j$ are connected if and only if $\si_i(t_j)\neq t_j$. In a sense, therefore, the complexity of this graph measures the extent to which the TGW algebra differs from a GWA. In \cite{Hartwig2006} this idea was further quantified by attaching multiplicities $(a_{ij},a_{ji})$ to each edge according to the degree of the minimal polynomial of $\si_i$ acting on $t_j$. It was also shown in \cite{Hartwig2006} that this results in a generalized Cartan matrix (GCM), that corresponding Serre-type relations hold, and that any symmetric GCM can occur. No examples of TGW algebras with non-symmetric GCM were known at the time. In Section \ref{sec:theGCM} we show that in fact any GCM can occur.

Another inspiration was the paper by A. Sergeev \cite{Sergeev2001}, in which he shows that certain infinite-dimensional primitive quotients of $U(\mathfrak{gl}_3)$ are TGW algebras. The corresponding simple $\mathfrak{gl}_3$-modules are completely pointed (multiplicity-free), and can be realized as eigenspaces of the Euler operator in simple weight modules over the Weyl algebra. The latter is a special case of the classification of such modules over $\mathfrak{sl}_n$ and $\mathfrak{sp}_{2n}$ given in \cite{BenBriLem1997}. Thus it is natural to ask for a generalization of Sergeev's construction to higher rank simple Lie algebras of types $A$ and $C$. We obtain this in Sections \ref{sec:special-linear-lie-algebras} and \ref{sec:symplectic-lie-algebras}.

We now describe the contents and main results of this article in more detail.

We assume throughout that $\K$ is a field of characteristic zero.
In Section \ref{sec:background} we collect some basic facts about TGW algebras.
A new result about quotients of TGW algebras that we will need is proved in Section \ref{sec:epimorphisms}.

The definition of the new TGW algebras $\mathcal{A}(\Ga)$ attached to multiquivers $\Ga$ is given in Section \ref{sec:AGamma}, along with some examples in Section \ref{sec:examples}. The relation to a previous family of TGW algebras studied in \cite{Hartwig2006} is given in Section \ref{sec:symmetric}.

In Section \ref{sec:diffops} we study a canonical representation $\varphi_\Gamma$ of $\mathcal{A}(\Ga)$ by differential operators. We determine the rank of the kernel of the incidence matrix of $\Ga$ as the number of connected components of $\Ga$ in \emph{equilibrium} (Theorem \ref{thm:rank-ker-gamma}), give a description of the centralizer of $R$ in $\mathcal{A}(\Ga)$ (Lemma \ref{lem:centralizer}) and prove that this centralizer is a maximal commutative subalgebra (Corollary \ref{cor:centralizer}). This is applied to obtain the first main result of the paper, which gives precise conditions under which $\varphi_\Ga$ is faithful.

{
\renewcommand{\thethm}{\ref{thm:faithfulness}}
\begin{thm}
Let $\Gamma$ be a multiquiver, $\gamma$ be its incidence matrix, $\mathcal{A}(\Ga)$ the corresponding TGW algebra, and $\varphi_\Gamma$ be the canonical representation by differential operators.
Then the following statements are equivalent:
\begin{enumerate}[{\rm (i)}]
\item $\varphi_\Gamma$ is faithful;
\item $\gamma$ is injective;
\item $R_E$ is a maximal commutative subalgebra of $\mathcal{A}(\Gamma)$;
\item No connected component of $\Ga$ is in equilibrium.
\end{enumerate}
\end{thm}
\addtocounter{thm}{-1}
} 

We are also interested in the question of how surjective $\varphi_\Gamma$ is. We say that it is \emph{locally surjective} if it maps each homogeneous component of $\mathcal{A}(\Ga)$ onto a homogeneous component of the Weyl algebra. The second main theorem of the paper gives a precise condition for $\varphi_\Ga$ to be locally surjective.

{
\renewcommand{\thethm}{\ref{thm:local-surjectivity}}
\begin{thm}
Let $\Gamma$ be a multiquiver and $\mathcal{A}(\Ga)$ be the corresponding TGW algebra. Let $\varphi_\Gamma$ be the canonical representation by differential operators.
Let $\bar\Ga$ denote the underlying undirected graph of $\Ga$.
Then the following two statements are equivalent:
\begin{enumerate}[{\rm (i)}]
\item $\varphi_\Ga$ is locally surjective;
\item $\bar\Gamma$ is acyclic.
\end{enumerate}
\end{thm}
\addtocounter{thm}{-1}
}

The third main result gives substance to the statement that $\varphi_\Ga$ is ``canonical''. In fact, we prove that the pair $(\mathcal{A}(\Ga), \varphi_\Ga)$ is universal in the following sense:
{
\renewcommand{\thethm}{\ref{thm:main}}
\begin{thm}
Let $\mathcal{A}=\TGWA{\mu}{R}{\si}{t}$ be any TGW algebra with index set denoted $V$, such that $R$ is a polynomial algebra $R=R_E=\K[u_e\mid e\in E]$ (for some index set $E$),
and $\mu$ is symmetric. Assume that
\[ 
\varphi:\mathcal{A}\to A_E(\K)
\] 
is a map of $R_E$-rings with involution, where $A_E(\K)$ is the Weyl algebra over $\K$ with index set $E$.
Then $\mathcal{A}$ is consistent and
there exists a multiquiver $\Gamma=(V,E,s,t)$ with vertex set $V$ and edge set $E$ and a map
\[\xi:\mathcal{A}\to\mathcal{A}(\Gamma)\]
of $\Z V$-graded $R_E$-rings with involution
such that the following diagram commutes: 
\[ 
\begin{aligned}
\begin{tikzcd}[ampersand replacement=\&,
               column sep=small]
\mathcal{A}
 \arrow{rr}{\varphi}
 \arrow{d}[swap]{\xi} 
 \& \& A_E(\K) \\ 
\mathcal{A}(\Gamma)
 \arrow{rru}[swap]{\varphi_{\Gamma}}
 \& \& 
\end{tikzcd}
\end{aligned}
\] 
Moreover, if $\varphi(X_i)\neq 0$ for each $i\in V$, then $\Gamma$ is uniquely determined and $\mu_{ij}=1$ for all $i,j\in V$. 
\end{thm}
\addtocounter{thm}{-1}
}

In the final part of the paper we address the relation between TGW algebras and universal enveloping algebras of simple Lie algebras. Section \ref{sec:theGCM} describes a simple algorithm which associates a Dynkin diagram $D(\Ga)$ to any multiquiver $\Ga$ in such a way that the GCM of $\mathcal{A}(\Ga)$ is exactly the one corresponding to $D(\Ga)$ (Theorem \ref{thm:commutative-diagram}). From this it easily follows that for any (not just symmetric) GCM $C$, there exists a TGW algebra $\mathcal{A}(\Ga)$ whose associated GCM is $C$. Another consequence is the existence of graded homomorphisms from the enveloping algebra of the positive and negative parts of the Kac-Moody algebra associated to $D(\Ga)$ into $\mathcal{A}(\Ga)$ (Theorem \ref{thm:serre-relations}).

In Theorem \ref{thm:primitivity} we show that $\mathcal{A}(\Ga)$ is a primitive ring, provided $\varphi_\Ga$ is faithful. This shows that perhaps it is more natural to expect that some primitive quotients of $U(\mathfrak{g})$, rather than $U(\mathfrak{g})$ itself, can be realized as TGW algebras.
Indeed, in Sections \ref{sec:special-linear-lie-algebras} and \ref{sec:symplectic-lie-algebras} we establish that for
$\mathfrak{sl}_{n+1}$ and $\mathfrak{g}=\mathfrak{sp}_{2n}$, any primitive quotient $U(\mathfrak{g})/J$, where $J$ is the annihilator of an infinite-dimensional simple completely pointed $\mathfrak{g}$-module, is graded isomorphic to TGW algebras of the form $\mathcal{A}(\widetilde{A_n})/\langle \mathbb{E}-\la\rangle$ where $\mathbb{E}$ is central and $\la\in\K$, respectively  $\mathcal{A}(\widetilde{C_n})$, and $\widetilde{C_n}$ and $\widetilde{A_n}$ are certain multiquivers whose Dynkin diagrams are $A_n$ and $C_n$ respectively. Our results clarify and generalize the type $A_2$ case considered in \cite{Sergeev2001} and give a direct relation between the representation $\varphi_\Ga$ and the classification of simple completely pointed $\mathfrak{g}$-modules in \cite{BenBriLem1997} using Weyl algebras.

This leads naturally to the following question: When is a primitive quotient  $U(\mathfrak{g})/J$ for a simple Lie algebra $\mathfrak{g}$ graded isomorphic to a TGW algebra? To get a complete answer we need to slightly generalize the notion of TGW algebra to allow $\si_i\si_j\neq\si_j\si_i$. This is our final main result:

{
\renewcommand{\thethm}{\ref{thm:primquotient}}
\begin{thm}
Assume that $\K$ is an algebraically closed field of characteristic zero.
Let $\mathfrak g$ be a finite-dimensional simple Lie algebra over $\K$ with
Serre generators $e_i,f_i$, $i=1,\dots,n$ and $J$ be a primitive ideal
of $U(\mathfrak g)$. The following conditions are equivalent:
\begin{enumerate}[{\rm (a)}]
\item There exists a not necessarily abelian TGW algebra $\TGWA{\mu}{R}{\si}{t}$ and a surjective homomorphism  
\[\psi: U(\mathfrak g)\to\TGWA{\mu}{R}{\si}{t}\]
with kernel $J$ such that $\psi(e_i)=X_i$, $\psi(f_i)=Y_i$;
\item There exists a simple completely pointed (multiplicity free)
weight $\mathfrak g$-module $M$ such that
\[\operatorname{Ann}_{U(\mathfrak g)} M=J.\] 
\end{enumerate}
\end{thm}
\addtocounter{thm}{-1}
}

Lastly, in Section \ref{sec:affine}, we prove that some primitive quotients of enveloping algebras
of affine Lie algebras can also be realized as TGW algebras.

\subsection*{Acknowledgements}
Part of this work was completed during the first author's visit at Stanford University. 
The first author is grateful to Daniel Bump, Vyacheslav Futorny, Johan \"Oinert,
Joanna Meinel and Apoorva Khare for interesting discussions related to this paper 
and to John Baez for Remark \ref{rem:markov} about Kolmogorov's criterion for detailed balance. 
Some results from this paper were reported at the School of Algebra in Salvador, Brazil in 2012.
The second author was supported by NSF grant DMS - 1303301.

\subsection{Notation}
Unless otherwise stated, $\K$ will denote a field of characteristic zero.
Rings and algebras are understood to be associative and unital. Subrings, subalgebras and homomorphisms of rings and algebras are unital.
The set of integers $x$ with $a\le x\le b$ is denoted $\iv{a}{b}$.

\subsection{Background on TGW algebras}
\label{sec:background}
We recall the definition of TGW algebras \cite{MazTur1999,MazTur2002} and some useful properties.

\subsubsection{TGW data}
\begin{dfn}[TGW datum]
Let $I$ be a set.
A \emph{twisted generalized Weyl datum over $\K$ with index set $I$} is a triple $(R,\si,t)$ where
\begin{itemize}
\item $R$ is a unital associative $\K$-algebra, called the \emph{base algebra},
\item $\si=(\si_i)_{i\in I}$ a sequence of commuting $\K$-algebra automorphisms of $R$,
\item $t=(t_i)_{i\in I}$ is a sequence of nonzero central elements of $R$.
\end{itemize} 
The cardinality of $I$ is called the \emph{rank} (or \emph{degree}) of $(R,\si,t)$.
\end{dfn}
Let $\Z I$ denote the free abelian group on the set $I$.
For $g=\sum g_i i\in\Z I$ we put $\si_g=\prod \si_i^{g_i}$. Then $g\mapsto\si_g$ defines an action of $\Z I$ on $R$ by $\K$-algebra automorphisms.
\begin{rem}
In Section \ref{sec:primitivequotients}, we will need to remove the assumption that $\si_i$ and $\si_j$ commute.
\end{rem}

\subsubsection{TGW constructions}

\begin{dfn}[TGW construction]
Let $I$ be a set and
\begin{itemize}
\item $(R,\si,t)$ be a TGW datum  over $\K$ with index set $I$,
\item $\mu$ be an $I\times I$-matrix
without diagonal, $\mu=(\mu_{ij})_{i\neq j}$,
with $\mu_{ij}\in\K\setminus\{0\}$.
\end{itemize}
The \emph{twisted generalized Weyl construction} associated to $\mu$ and $(R,\si,t)$, denoted $\TGWC{\mu}{R}{\si}{t}$, is defined as the free $R$-ring on the set $\bigcup_{i\in I}\{X_i,Y_i\}$ modulo the two-sided ideal generated by the following elements:
\begin{subequations}\label{eq:tgwarels}
\begin{alignat}{3}
\label{eq:tgwarels1}
X_ir  &-\si_i(r)X_i,  &\quad Y_ir&-\si_i^{-1}(r)Y_i, 
 &\quad \text{$\forall r\in R,\, i\in I$,} \\
\label{eq:tgwarels2}
Y_iX_i&-t_i, &\quad X_iY_i&-\si_i(t_i),
 &\quad \text{$\forall i\in I$,} \\
\label{eq:tgwarels3}
&&\quad X_iY_j&-\mu_{ij}Y_jX_i,
 &\quad \text{$\forall i,j\in I,\, i\neq j$.}
\end{alignat}
\end{subequations}
\end{dfn}
The ring $\TGWC{\mu}{R}{\si}{t}$ has a $\Z I$-gradation $\TGWC{\mu}{R}{\si}{t}=\bigoplus_{d\in\Z I} \TGWC{\mu}{R}{\si}{t}_d$ given by requiring 
\begin{equation}\label{eq:TGWA-gradation}
\begin{gathered}
\deg X_i= i,\quad \deg Y_i=(-1)i,\quad \forall i\in I,\\
\deg r=0, \quad \forall r\in R.
\end{gathered}
\end{equation}

\subsubsection{TGW algebras}

Let $\TGWI{\mu}{R}{\si}{t}\subseteq \TGWC{\mu}{R}{\si}{t}$ be 
the sum of all graded ideals $J\subseteq \TGWC{\mu}{R}{\si}{t}$ 
such that \[\TGWC{\mu}{R}{\si}{t}_0\cap J=\{0\},\]
where $\TGWC{\mu}{R}{\si}{t}_0$ denotes the degree zero component with respect to the $\Z I$-gradation \eqref{eq:TGWA-gradation}.
It is easy to see that $\TGWI{\mu}{R}{\si}{t}$ is the unique maximal graded ideal of $\TGWC{\mu}{R}{\si}{t}$ having zero intersection with $\TGWC{\mu}{R}{\si}{t}_0$.

\begin{dfn}[TGW algebra]
The \emph{twisted generalized Weyl algebra} $\TGWA{\mu}{R}{\si}{t}$ associated to $\mu$ and $(R,\si,t)$ is
defined as the quotient $\TGWA{\mu}{R}{\si}{t}:=\TGWC{\mu}{R}{\si}{t} / \TGWI{\mu}{R}{\si}{t}$.
\end{dfn}
Since $\TGWI{\mu}{R}{\si}{t}$ is graded,
$\TGWA{\mu}{R}{\si}{t}$ inherits a $\Z I$-gradation from $\TGWC{\mu}{R}{\si}{t}$.
The images in $\TGWA{\mu}{R}{\si}{t}$ of the elements $X_i, Y_i$ will also be denoted by $X_i, Y_i$.

\subsubsection{Example: The Weyl algebra}\label{sec:WeylAlgebraExample}
The Weyl algebra over $\K$ with index set $I$, denoted $A_I(\K)$, is the $\K$-algebra generated by $\{x_i,y_i\mid i\in I\}$ subject to defining relations
\[[x_i,x_j]=[y_i,y_j]=[y_i,x_j]-\delta_{ij}=0,\quad\forall i,j\in I.\]
There is a $\K$-algebra isomorphism $\TGWA{\mu}{R}{\tau}{u}\to A_I(\K)$ where $\mu_{ij}=1$ for all $i\neq j$, $R=\K[u_i\mid i\in I]$, $\tau=(\tau_i)_{i\in I}$, $\tau_i(u_j)=u_j-\delta_{ij}$, $u=(u_i)_{i\in I}$,
 given by $X_i\mapsto x_i$, $Y_i\mapsto y_i$, $u_i\mapsto y_ix_i$.

\subsubsection{Reduced and monic monomials}
A \emph{monic monomial} in a TGW algebra is any product of elements from the set $\bigcup_{i\in I}\{X_i,Y_i\}$.
A \emph{reduced monomial} is a monic monomial of the form
$Y_{i_1}\cdots Y_{i_k} X_{j_1}\cdots X_{j_l}$
where $\{i_1,\ldots,i_k\}\cap \{j_1,\ldots,j_l\}=\emptyset$.
The following statement is easy to check.
\begin{lem}{\cite[Lem.~3.2]{Hartwig2006}}
\label{lem:monomials}
$\TGWA{\mu}{R}{\si}{t}$ is generated as a left (and as a right) $R$-module by the reduced monomials.
\end{lem}
Since a TGW algebra $\TGWA{\mu}{R}{\si}{t}$ is a quotient of an $R$-ring, it is an $R$-ring itself with a natural map $\rho:R\to \TGWA{\mu}{R}{\si}{t}$.
By Lemma \ref{lem:monomials}, the degree zero component of $\TGWA{\mu}{R}{\si}{t}$ (with respect to the $\Z I$-gradation) is equal to the image of $\rho$.

\subsubsection{Regularity and consistency}
\begin{dfn}[Regularity]
A TGW datum $(R,\si,t)$ is called \emph{regular} if $t_i$
is regular (i.e. not a zero-divisor) in $R$ for all $i\in I$.
\end{dfn}
Due to relation \eqref{eq:tgwarels2}, the canonical map $\rho:R\to\TGWC{\mu}{R}{\si}{t}$ is not guaranteed to be injective, and indeed sometimes it is not \cite{FutHar2012b}. It is injective if and only if the map $R\to\TGWA{\mu}{R}{\si}{t}$ is injective. 
\begin{dfn}[$\mu$-Consistency]
A TGW datum $(R,\si,t)$ is \emph{$\mu$-consistent} if the canonical map $\rho:R\to \TGWA{\mu}{R}{\si}{t}$ is injective.
\end{dfn}
By abuse of language we sometimes say that the TGW algebra $\TGWA{\mu}{R}{\mu}{t}$ is consistent if $(R,\si,t)$ is $\mu$-consistent.

\begin{thm}{\cite{FutHar2012b}}\label{thm:consistency}
If $(R,\si,t)$ is a regular TGW datum, and $\mu=(\mu_{ij})_{i\neq j}$ with $\mu_{ij}\in\K\setminus\{0\}$, then $(R,\si,t)$ is $\mu$-consistent iff
\begin{subequations}\label{eq:consistency_rels}
\begin{align}\label{eq:consistency_rel1}
\si_i\si_j(t_it_j)&=\mu_{ij}\mu_{ji}\si_i(t_i)\si_j(t_j),\quad\forall i\neq j;\\
\si_i\si_k(t_j)t_j&=\si_i(t_j)\si_k(t_j),\quad \forall i\neq j\neq k\neq i.
\end{align}
\end{subequations}
\end{thm}
That relation \eqref{eq:consistency_rel1} is necessary for consistency of a regular TGW datum was known already in \cite{MazTur1999},\cite{MazTur2002}.
If $(R,\si,t)$ is not regular, sufficient and necessary conditions for $\mu$-consistency are not known.

For consistent TGW algebras one can characterize regularity as follows:
\begin{thm}{\cite[Thm.~4.3]{HarOin2013}}
\label{thm:regularity}
Let $A=\TGWA{\mu}{R}{\si}{t}$ be a consistent TGW algebra. Then the following are equivalent
\begin{enumerate}[{\rm (i)}]
\item $(R,\si,t)$ is regular;
\item Each monic monomial in $A$ is non-zero and generates a free left (and right) $R$-module of rank one;
\item $A$ is \emph{regularly graded}, i.e. for all $g\in\Z I$, there exists a nonzero, regular element in $A_g$;
\item If $a\in A$ is a homogeneous element such that $bac=0$ for some monic monomials $b,c\in A$, then $a=0$.
\end{enumerate}
\end{thm}

\subsubsection{Non-degeneracy of the gradation form}
For a group $G$, any $G$-graded ring $A=\bigoplus_{g\in G}A_g$
can be equipped with a $\Z$-bilinear form $f:A\times A\to A_e$ called the \emph{gradation form}, defined by
\begin{equation}
f(a,b)=\mathfrak{p}_e(ab)
\end{equation}
where $\mathfrak{p}_e$ is the projection $A\to A_e$ along the direct sum $\bigoplus_{g\in G} A_g$, and $e\in G$ is the neutral element.

\begin{thm}{\cite[Cor.~3.3]{HarOin2013}}
\label{thm:nondeg}
The ideal $\TGWI{\mu}{R}{\si}{t}$ is equal to the radical of the gradation form on $\TGWC{\mu}{R}{\si}{t}$ (with respect to the $\Z I$-gradation), and thus the gradation form on $\TGWA{\mu}{R}{\si}{t}$ is non-degenerate.
\end{thm}

\subsubsection{\texorpdfstring{$R$}{R}-rings with involution}
\label{sec:R-ring-with-involution}
We need the concept of an $R$-ring with involution, defined as follows.
\begin{dfn}\label{dfn:Rring_with_involution} 
\begin{enumerate}[{\rm (i)}]
\item An \emph{involution} on a ring $A$ is a $\Z$-linear map $\ast:A\to A, a\mapsto a^\ast$ satisfying $(ab)^\ast=b^\ast a^\ast$, $(a^\ast)^\ast=a$ for all $a,b\in A$.
\item Let $R$ be a ring.
An \emph{$R$-ring with involution} is a ring $A$ equipped with a ring homomorphism $\rho_A:R\to A$ and an involution $\ast:A\to A$ such that $\rho_A(r)^\ast=\rho_A(r)$ for all $r\in R$.
\item 
If $A$ and $B$ are two $R$-rings with involution, then a \emph{map of $R$-rings with involution} is a ring homomorphism $k:A\to B$ such that $k\circ \rho_A=\rho_B$ and $k(a^\ast)=k(a)^\ast$ for all $a\in A$.
\end{enumerate}
\end{dfn}
Any TGW algebra $A=\TGWA{\mu}{R}{\si}{t}$ with index set $I$ for which $\mu_{ij}=\mu_{ji}$ for all $i,j$, can be equipped with an involution $\ast$ given by $X_i^\ast=Y_i,\, Y_i^\ast=X_i\;\forall i\in I$, $r^\ast=r\;\forall r\in R$. Together with the canonical map $\rho:R\to A$ this turns $A$ into an $R$-ring with involution. In particular we regard the Weyl algebra $A_I(\K)$ as an $R$-ring with involution in this way, where $R=\K[u_i\mid i\in I]$ as in Section \ref{sec:WeylAlgebraExample}.

\subsection{A proposition about epimorphisms of TGW algebras}
\label{sec:epimorphisms}

The following proposition about a quotient construction will be used in Section \ref{sec:special-linear-lie-algebras} and is a variation of \cite[Thm.~4.1]{FutHar2012b}. 
Let $A=\mathcal{A}_\mu(R,\si,t)$ be a TGW algebra with index set $I$. We say that an ideal $J$ of $R$ is \emph{$\Z I$-invariant} if $\si_d(J)\subseteq J$ for all $d\in \Z I$. Then $J$ satisfies
\begin{equation}\label{eq:AJAJAJA}
AJ=AJA=JA
\end{equation}
Indeed, \[AJA=\sum_{g,h\in \Z I} A_gJA_h \subseteq \sum_{g,h\in \Z I} \si_g(J)A_gA_h\subseteq \sum_{g,h\in \Z I} JA_hA_g=JA.\] 
The inclusion $AJA\subseteq AJ$ is proved similarly. The reverse inclusions are trivial since $A$ is unital.

\begin{prp}\label{prp:quotients}
Let $A=\mathcal{A}_\mu(R,\si,t)$ be a consistent and regular TGW algebra with index set $I$ and let $J$ be a $\Z I$-invariant prime ideal of $R$ such that $t_i\notin J$ for all $i\in I$. Assume that $A_g$ is cyclic as a left (equivalently, right) $R$-module for all $g\in \Z I$. 
Let $\bar A=\mathcal{A}_\mu(\bar R, \bar\si, \bar t)$ where
$\bar R=R/J$, $\bar\si=(\bar\si_i)_{i\in I}$, $\bar\si_i(r+J)=\si_i(r)+J$, $\bar t=(t_i+J)_{i\in I}$. Then the kernel of the surjective homomorphism
\begin{equation}
\begin{aligned}
Q_J:A &\to \bar A\\ 
X_i &\mapsto X_i \\ 
Y_i &\mapsto Y_i \\ 
r &\mapsto r+J
\end{aligned}
\end{equation}
is equal to $AJA$.
\end{prp}
\begin{proof} 
By \cite[Lem.~3.3]{FutHar2012b}, $\ker Q_J$ is equal to the sum of all graded ideals of $A$ whose intersection with $R$ is contained in $J$. In particular 
\begin{equation}\label{eq:RkerQJJ}
R\cap \ker Q_J\subseteq J.
\end{equation}
By \eqref{eq:AJAJAJA}, $(AJA)\cap R=(AJ)\cap R = RJ = J$,
hence $AJA\subseteq \ker Q_J$.

For the converse, let $d\in \Z I$ and $a\in A_d\cap\ker Q_J$.
For $g\in \Z I$, let $u_g$ be a generator for $A_g$ as a left $R$-module. 
We have $a=ru_d$ for some $r\in R$. We will show that $r\in J$. First we show that $u_du_{-d}\notin J$.
Let $m\in A_d$ and $m'\in A_{-d}$ be monic monomials
(products of elements from $\{X_i\}_{i\in I}\cup\{Y_i\}_{i\in I}$). By the TGW relations \eqref{eq:tgwarels}, $mm'$ is a product of elements of the form $\si_g(t_i)$. Since $J$ is $\Z I$-invariant, prime, and does not contain $t_i$, it follows that $mm'\notin J$.
On the other hand, $m=su_d$ and $m'=s'u_{-d}$ for some $s,s'\in R$. Thus
$mm'=s\si_d(s')u_du_{-d}$. Since $J$ is prime, it follows that $u_du_{-d}\notin J$.
We have $ru_du_{-d}=au_{-d}\in R\cap \ker Q_J\subseteq J$ by \eqref{eq:RkerQJJ}, which implies that $r\in J$ since $J$ is prime. Thus $a=ru_d\in JA=AJA$. Since $d$ was arbitrary and $Q_J$ is graded, this proves that $\ker Q_J\subseteq AJA$.
\end{proof}

\section{A family of TGW algebras parametrized by multiquivers}
\label{sec:AGamma}

\subsection{Multiquivers}
\label{sec:DirectedGraphs}

\begin{dfn}
A \emph{multiquiver} is a quadruple $\Gamma=(V,E,s,t)$ where
\begin{itemize}
\item $V$ and $E$ are sets,
\item $s,t:E\to (V\times\Z_{>0})\cup\{\emptyset\}$ are two functions,
\end{itemize}
subject to the following two conditions:
\begin{align}
\text{(No loops):}&\quad\text{If $s(e)=(u,m)$ and $t(e)=(v,n)$ then $u\neq v$.} \\ 
\text{(Locally finiteness):}&\quad\text{For each $v\in V$, $s^{-1}(\{v\}\times\Z_{>0})$ and $t^{-1}(\{v\}\times\Z_{>0})$ are finite sets.}
\end{align}
The elements of $V$ and $E$ are called \emph{vertices} and \emph{edges} respectively.
Let $e\in E$. If $s(e)=(v,n)$ (respectively $t(e)=(v,n)$) for some $v\in V$, $n\in \Z_{>0}$, then $v$ is called the \emph{source} (respectively \emph{target}) of $e$, $v$ is said to be \emph{incident} to $e$ and $n$ is the \emph{outgoing} (respectively \emph{incoming}) \emph{multiplicity} of $e$ at $v$.
If $s(e)=\emptyset$ (respectively $t(e)=\emptyset$) we say that $e$ has no source (respectively no target).
If $e$ has a target but no source, or a source but no target, $e$ is a \emph{connected leaf} (or a \emph{half-edge}).
If $e$ has neither a source nor a target, $e$ is a \emph{disconnected leaf} (or a \emph{loose edge}).
We depict an edge $e$ with $s(e)=(v_1,a)$ and $t(e)=(v_2,b)$ as follows:
\[ 
\begin{tikzpicture}[baseline=(1.base)]
\node[vertex] (1) at (0,0) [label=below:$v_1$] {};
\node[vertex] (2) at (2cm,0) [label=below:$v_2$] {};
\draw[thick,directed] (1) -- node[above] {$e$} node[auto,very near start] {$a$} node[auto,very near end] {$b$} (2);
\end{tikzpicture}
\]
If $a=1$ or $b=1$ we sometimes omit the corresponding incidence multiplicity from the diagram.

A multiquiver $\Ga=(V,E,s,t)$ is called a \emph{simple quiver} if $\Ga$ has no leaves and has at most one edge between any two vertices.

A multiquiver $\Ga=(V,E,s,t)$ without leaves is called \emph{symmetric} (or \emph{equally valued)} if, in each edge, the outgoing and incoming multiplicities coincide.
\end{dfn}

\begin{rem}\label{rem:simple-quiver}
In a simple quiver $\Ga=(V,E,s,t)$ one may re-interpret each edge $e\in E$, say $s(e)=(i,v_{ij})$ and $t(e)=(j,v_{ji})$, as a set of two labeled directed edges, one going from $i$ to $j$ with label $v_{ij}$, and the other going from $j$ to $i$ with label $v_{ji}$. In this way we obtain a directed graph (digraph) $Q(\Ga)$ with edges labeled by positive integers. However, some information is lost; switching the direction of $e$ in $\Ga$, i.e. replacing it by $e'$ with $s(e')=(j,v_{ji})$, $t(e')=(i,v_{ij})$, yields the same set of two labeled edges in $Q(\Ga)$. 
\end{rem}

\subsection{Incidence matrix}
For a set $X$ we let $\Z X$ denote the free abelian group on $X$. To each multiquiver $\Gamma=(V,E,s,t)$ we associate a group homomorphism $\gamma\in\Hom_\Z(\Z V, \Z E)$ as follows:
\begin{equation}
\label{eq:gamma-definition}
\gamma(v)=\sum_{\substack{e\in E,\, n\in\Z_{>0}\\t(e)=(v,n)}}n\cdot e - \sum_{\substack{e\in E,\, n\in\Z_{>0}\\s(e)=(v,n)}}n\cdot e, \qquad \forall v\in V.
\end{equation}
Since $\Gamma$ is locally finite, $\gamma(v)$ is a finite linear combination of elements of $E$, hence $\gamma(v)\in \Z E$.
We call $\gamma$ the \emph{incidence matrix} of $\Gamma$.
If $V$ and $E$ are nonempty, we can identify
 $\ga$ with the matrix $(\ga_{ev})_{e\in E, v\in V}$ given by
\[\gamma(v)=\sum_{e\in E}\gamma_{ev} e,
\qquad \forall v\in V.\]
It satisfies the following two properties:
\begin{gather}
\label{eq:condition_C}
\text{Every column $(\gamma_{ev})_{e\in E}$ has at most finitely many nonzero elements.} \\ 
\label{eq:condition_M}
\text{Every row $(\gamma_{ev})_{v\in V}$ contains at most one positive and at most one negative element.}
\end{gather}
Conversely, if $A=(a_{ij})_{i\in I,j\in J}$ is any matrix with integer entries satisfying conditions \eqref{eq:condition_C} and \eqref{eq:condition_M} there exists a unique multiquiver $\Gamma_A$ whose incidence matrix is $A$. Namely, $\Gamma_A=(J,I,s,t)$ where for any $i\in I$,
\begin{align*}
s(i)&=
\begin{cases}
(j, |a_{ij}|)& \text{if $a_{ij}<0$ for some (unique) $j\in J$,}\\
\emptyset & \text{otherwise,}
\end{cases}\\ 
t(i)&=
\begin{cases}
(j,a_{ij})& \text{if $a_{ij}>0$ for some (unique) $j\in J$,}\\
\emptyset & \text{otherwise.}
\end{cases}
\end{align*}
Thus there is a one-to-one correspondence between the set of multiquivers and the set of homomorphisms $\Hom_\Z(\Z V, \Z E)$ whose matrix satisfies 
conditions \eqref{eq:condition_C} and \eqref{eq:condition_M}.
Connected leaves correspond to rows with exactly one non-zero element while disconnected leaves correspond to rows with only zeroes.
Two multiquivers are \emph{isomorphic} if their incidence matrices coincide up to permutation of rows and columns.

\subsection{Construction of the TGW algebra \texorpdfstring{$\mathcal{A}(\Gamma)$}{A(Gamma)}}

Let $\Gamma=(V,E,s,t)$ be a multiquiver, and $\gamma=(\gamma_{ev})_{e\in E, v\in V}$ be its incidence matrix. Define
\begin{subequations}
\begin{align}
\mu&=(\mu_{ij})_{i,j\in V}, \quad\mu_{ij}=1, \quad\forall i,j\in V,\\
R_E &=\K[u_e\mid e\in E]\quad\text{(polynomial algebra),}\\
\label{eq:sigma-definition}
\sigma^\Gamma &=(\sigma_v)_{v\in V},\quad \sigma_v(u_e)=u_e-\gamma_{ev},\quad\forall v\in V, e\in E,\\
\label{eq:t-definition}
t^\Gamma &=(t_v)_{v\in V}, \quad t_v = \prod_{\substack{e\in E\\ \text{$v$ incident to $e$}}} u_{ev},
\end{align}
where
\begin{equation}
u_{ev}=
\begin{cases}
 u_e(u_e+1)\cdots (u_e+n-1),&\text{if $t(e)=(v,n)$}, \\ 
 (u_e-1)(u_e-2)\cdots (u_e-n),&\text{if $s(e)=(v,n)$}.
\end{cases}
\end{equation}
\end{subequations}
We define $\mathcal{A}(\Gamma)$ to be the twisted generalized Weyl algebra $\mathcal{A}_\mu(R_E,\sigma^\Ga,t^\Ga)$ with index set $V$. 
$\mathcal{A}(\Gamma)$ is a $\Z V$-graded algebra. Since $\mu$ is symmetric, $\mathcal{A}(\Gamma)$ is also an $R_E$-ring with involution 
(see Section \ref{sec:R-ring-with-involution}).

\begin{rem}
In terms of the matrix $\gamma$ we have
\begin{equation}
t_v=\prod_{e\in E} u_{ev}, \qquad 
u_{ev}=
\begin{cases}
 u_e(u_e+1)\cdots (u_e+\ga_{ev}-1),& \ga_{ev}> 0, \\ 
 1, &\ga_{ev}=0,\\
 (u_e-1)(u_e-2)\cdots (u_e-|\ga_{ev}|),& \ga_{ev}<0.
\end{cases}
\end{equation}
\end{rem}

\begin{thm} \label{thm:Gamma-solves-consistency}
For any multiquiver $\Gamma=(V,E,s,t)$, the TGW datum $(R_E,\si^\Gamma,t^\Gamma)$ satisfies the TGW consistency conditions \eqref{eq:consistency_rels} with $\mu_{ij}=1$ for $i,j\in V$.
\end{thm}
\begin{proof}
By \eqref{eq:t-definition} it suffices to check that for each $e\in E$,
\[\si_i(u_{ei})\si_j(u_{ej})=\si_i\si_j(u_{ei}u_{ej}),
\quad \forall i,j\in V, i\neq j,\]
and
\[\si_i\si_k(u_{ej})u_{ej}=\si_i(u_{ej})\si_k(u_{ej}),
\quad \forall i,j,k\in V, i\neq j\neq k\neq i.\] 
These identities are easy to verify, using \eqref{eq:sigma-definition} and
Property \eqref{eq:condition_M} of the matrix $(\ga_{ev})_{e\in E,v\in V}$.
\end{proof}

\begin{cor} \label{cor:consistency_of_AGamma}
For any multiquiver $\Gamma$, the TGW algebra $\mathcal{A}(\Gamma)$ is consistent,
i.e. the canonical map $\rho:R_E\to\mathcal{A}(\Gamma)$ is injective.
\end{cor}
\begin{proof}
Immediate by Theorem \ref{thm:Gamma-solves-consistency} and Theorem \ref{thm:consistency}.
\end{proof}

We will henceforth use $\rho$ to identify $R_E$ with its image in $\mathcal{A}(\Gamma)$. By Lemma \ref{lem:monomials}, $\mathcal{A}(\Gamma)_0=R_E$.

\begin{cor}\label{cor:domain}
$\mathcal{A}(\Gamma)$ is a domain.
\end{cor}
\begin{proof}
Since $R_E$ is a domain, this follows from Corollary \ref{cor:consistency_of_AGamma}
and \cite[Prop.~2.9]{FutHar2012a}.
\end{proof}

\begin{rem}
Different multiquivers $\Gamma$ can yield isomorphic algebras $\mathcal{A}(\Gamma)$. However, if $\Gamma_1=(V,E,s_1,t_1)$ and $\Gamma_2=(V,E,s_2,t_2)$ are two multiquivers with the same underlying vertex and edge sets, then $\mathcal{A}(\Gamma_1)$ and $\mathcal{A}(\Gamma_2)$ are isomorphic as $\Z V$-graded $R_E$-rings if and only if $\Gamma_1$ and $\Gamma_2$ are isomorphic.
\end{rem}

\section{Examples}
\label{sec:examples}
We list a number of examples of multiquivers, their incidence matrices and describe the corresponding algebras $\mathcal{A}(\Ga)$.

\begin{enumerate}[{\rm (1)}]
\item (Trivial case) $V=\emptyset$, $E=\emptyset$.
Then $\Z V=\Z E=\{0\}$, $\ga$ is the zero map $\{0\}\to\{0\}$ and $\mathcal{A}(\Gamma)\simeq \K$.

\item (Laurent polynomial algebra) $V$ arbitrary, $E=\emptyset$.
Then $\ga$ is the zero map $\Z V\to \{0\}$.
\[
\Ga:\;
\begin{tikzpicture}[baseline=(etc.base)]
\node[vertex] (1) at (0,0) {};
\node[vertex] (2) at (.5,0) {};
\node[vertex] (3) at (1,0) {};
\node[vertex] (4) at (2,0) {};
\node (etc) at (1.5,0) {$\cdots$};
\end{tikzpicture}
\qquad\qquad
\ga=0:\Z V\to \{0\}
\] 
We have $t_v=1$ for each $v\in V$ and $\mathcal{A}(\Gamma)\simeq \K[X_v,X_v^{-1}\mid v\in V]$, a Laurent polynomial algebra.

\item (Polynomial algebra) $V=\emptyset$, $E$ arbitrary.
Then $\ga$ is the zero map $\{0\}\to\Z E$.
\[
\Ga:\;
\begin{tikzpicture}[baseline=(etc.base)]
\node (1) at (0,0) {};
\node (2) at (.5,0) {};
\node (3) at (1,0) {};
\node (4) at (2,0) {};
\node (leaf1) at (0,1)    {};
\node (leaf2) at (.5,1)  {};
\node (leaf3) at (1,1)  {};
\node (leaf4) at (2,1)  {};
\node (etc) at (1.5,.5) {$\cdots$};
\draw[leaf] (1) -- (leaf1);
\draw[leaf] (2) -- (leaf2);
\draw[leaf] (3) -- (leaf3);
\draw[leaf] (4) -- (leaf4);
\end{tikzpicture}
\qquad\qquad
\ga=0:\{0\} \to \Z E
\]
We have $\mathcal{A}(\Gamma)\simeq R_E=\K[u_e\mid e\in E]$, a polynomial algebra.

\item (Weyl algebra)
Let $I$ be any set, and $V=E=I$, $s(i)=\emptyset$, $t(i)=(i,1)$ for all $i\in I$. Then $\ga$ is the identity map $\Z I\to \Z I$ and $\mathcal{A}(\Ga)$ is isomorphic to the Weyl algebra $A_I(\K)$.
\[
\Ga:\;
\begin{tikzpicture}[baseline=(X.base)]
\node (X) at (0,.4) {};
\node[vertex] (1) at (0,0) {};
\node[vertex] (2) at (.5,0) {};
\node[vertex] (3) at (1,0) {};
\node[vertex] (4) at (2,0) {};
\node (leaf1) at (0,1)    {};
\node (leaf2) at (.5,1)  {};
\node (leaf3) at (1,1)  {};
\node (leaf4) at (2,1)  {};
\node (etc) at (1.5,0) {$\cdots$};
\draw[leaf] (leaf1) -- (1);
\draw[leaf] (leaf2) -- (2);
\draw[leaf] (leaf3) -- (3);
\draw[leaf] (leaf4) -- (4);
\end{tikzpicture}
\qquad\qquad
\ga=
\begin{bmatrix}
1 & 0 & \cdots & 0 \\
0 & 1 & \cdots & 0 \\
\vdots & \vdots & \ddots & \vdots \\
0 & 0 & \cdots & 1
\end{bmatrix}
\]

\item (An example of type $A_2$)
\[
\Ga:\;  
\begin{tikzpicture}[baseline=(X.base)]
\node (X) at (0,-0.1) {};
\node[vertex] (1) at (0,0) [label=below:$v_1$] {};
\node[vertex] (2) at (2cm,0) [label=below:$v_2$] {};
\draw[edge] (1) -- (2);
\end{tikzpicture}
\qquad\qquad
\ga =
\begin{bmatrix}
-1 & 1
\end{bmatrix}
\] 
Then $\mathcal{A}(\Ga)$ is graded isomorphic to the TGW algebra defined by Mazorchuk and Turowska  \cite[Ex.~1.3]{MazTur1999}. Therefore, by \cite[Ex.~6.3]{Hartwig2010}, $\mathcal{A}(\Ga)$ is isomorphic to the algebra with generators $u, X_1,X_2,Y_1,Y_2$ and defining relations
\[
\begin{aligned}
X_1u&=(u+1)X_1, & X_2u&=(u-1)X_2, \\ 
Y_1u&=(u-1)Y_1, & Y_2u&=(u+1)Y_2, \\ 
Y_1X_1&=X_2Y_2=u-1, & X_1Y_2&=Y_2X_1,\\
Y_2X_2&=X_1Y_1=u, &  X_2Y_1&=Y_1X_2.
\end{aligned}
\quad
\begin{aligned}
X_1^2X_2-2X_1X_2X_1+X_2X_1^2&=0, \\ 
X_2^2X_1-2X_2X_1X_2+X_1X_2^2&=0, \\ 
Y_1^2Y_2-2Y_1Y_2Y_1+Y_2Y_1^2&=0, \\ 
Y_2^2Y_1-2Y_2Y_1Y_2+Y_1Y_2^2&=0.
\end{aligned}
\]

\item (An example related to $\mathfrak{gl}_{n+1})$
\[ 
\Ga:\;
\begin{tikzpicture}[baseline=(etc.base)]
\node (start) at (0,0) {};
\node[vertex] (1) at (1,0) [label=below:$1$] {};
\node[vertex] (2) at (2,0) [label=below:$2$] {};
\node[vertex] (3) at (3,0) [label=below:$3$] {};
\node[vertex] (N) at (4,0) [label=below:$n$] {};
\node (end) at (5,0) {};
\node (etc) at (3.5,0) {$\cdots$};
\draw[leaf] (start) -- (1);
\draw[edge] (1) -- (2);
\draw[edge] (2) -- (3);
\draw[leaf] (N) -- (end);
\end{tikzpicture}
\qquad\qquad
\ga=\begin{bmatrix}
1  &    &        & &\\
-1 & 1  &        & &\\
   & -1 & \ddots       &  \\
   &    & \ddots & \ddots & \\
   &    &        & -1 & 1 \\
   &    &        & &  -1
\end{bmatrix}
\] 
Then $\mathcal{A}(\Ga)$ is a graded homomorphic image of $U(\mathfrak{gl_{n+1}})$, under $E_{i,i+1}\mapsto X_i$, $E_{i+1,i}\mapsto Y_i$, $E_{jj}\mapsto u_j$ for $i=1,2,\ldots,n$ and $j=1,2,\ldots,n+1$. This generalizes the $\mathfrak{gl}_3$ case considered by A. Sergeev \cite{Sergeev2001}.
See Section \ref{sec:special-linear-lie-algebras} for more details.

\item (An example related to $\mathfrak{sp}_{2n}$)
\[
\Ga:\;\; 
\begin{tikzpicture}[baseline=(etc.base)]
\node (start) at (0,0) {};
\node[vertex] (1) at (1,0) [label=below:$1$] {};
\node[vertex] (2) at (2,0) [label=below:$2$] {};
\node[vertex] (3) at (3,0) [label=below:$3$] {};
\node[vertex] (N-2) at (4,0) [label=below:$n-2$] {};
\node[vertex] (N-1) at (5,0) [label=below:$n-1$] {};
\node[vertex] (N) at (6,0) [label=below:$n$] {};
\node (etc) at (3.5,0) {$\cdots$};

\draw[leaf] (start) -- (1);
\draw[edge] (1) -- (2);
\draw[edge] (2) -- (3);
\draw[edge] (N-2) -- (N-1);
\draw[edge] (N-1) -- node[auto,very near start] {$1$} node[auto,very near end] {$2$} (N);
\end{tikzpicture}
\qquad\qquad 
\ga=\begin{bmatrix}
1  &    &        & && \\
-1 & 1  &        & &&\\
   & -1 & \ddots       &  &&\\
   &    & \ddots & \ddots &&\\
   &    &        & -1 & 1 & \\
   &    &        &    &-1 & 2
\end{bmatrix}
\] 
Then $\mathcal{A}(\Ga)$ is a graded homomorphic image of $U(\mathfrak{sp}_{2n})$. See Section \ref{sec:symplectic-lie-algebras} for more details.

\item (Algebras associated to symmetric GCMs)
Let $\Ga=(V,E,s,t)$ be a symmetric simple quiver.
Then $\mathcal{A}(\Ga)$ is graded isomorphic to a quotient of the algebra $\mathcal{T}(C)$ from \cite{Hartwig2010}, where $C$ is the (symmetric) generalized Cartan matrix of the Dynkin diagram associated to $\Ga$. See Section \ref{sec:symmetric} for details.

\end{enumerate}

\section{Representation by differential operators}
\label{sec:diffops}

\subsection{The homomorphism \texorpdfstring{$\varphi_\Gamma$}{phiGamma}}
\label{sec:homomorphism-varphi}

Let $\Gamma=(V,E,s,t)$ be a multiquiver. In this section we show how the algebras $\mathcal{A}(\Gamma)$ can be naturally represented by differential operators.
Let $A_E(\K)$ be the Weyl algebra with index set $E$.
For $k=\sum_{e\in E} k_e e\in \Z E$, define $z^k\in A_E(\K)$ by
\begin{equation}\label{eq:zk_notation}
z^k=\prod_{e\in E}z_e^{(k_e)},\qquad \text{where }z_e^{(p)}=\begin{cases}x_e^p& p\ge 0,\\ y_e^{-p}& p<0.\end{cases}
\end{equation}

\begin{thm}\label{thm:map_to_Weyl_algebra}
Let $\Gamma=(V,E,s,t)$ be a multiquiver and $\gamma\in\Hom_\Z(\Z V, \Z E)$ be the corresponding group homomorphism \eqref{eq:gamma-definition}.
There exists a map of $R_E$-rings with involution
\begin{subequations}
\label{eq:varphi_def}
\begin{gather}
\varphi_\Gamma:\mathcal{A}(\Gamma)\longrightarrow A_E(\K),\\
\intertext{uniquely determined by}
\label{eq:varphi_def_1}
\varphi(X_v)=z^{\gamma(v)},\quad
\varphi(Y_v)=z^{-\gamma(v)},\quad \forall v\in V \\ 
\label{eq:varphi_def_2}
\varphi(u_e)=y_ex_e,\quad\forall e\in E.
\end{gather}
\end{subequations}
\end{thm}
\begin{proof}
Put $(R,\si,t)=(R_E,\si^\Ga,t^\Ga)$. Let $\mathcal{X}=\cup_{v\in V}\{X_v,Y_v\}$ and define $\varphi:\mathcal{X} \to A_E(\K)$ by \eqref{eq:varphi_def_1}-\eqref{eq:varphi_def_2}.
Extend $\varphi$ uniquely to a homomorphism of $R$-rings $\varphi:F_R(\mathcal{X})\to A_E(\K)$, where $F_R(\mathcal{X})$ is the free $R$-ring on the set $\mathcal{X}$.
Using \eqref{eq:condition_M}, one verifies that the ideal in $F_R(\mathcal{X})$ generated by the elements \eqref{eq:tgwarels} is contained in the kernel of $\varphi$. Thus we get an induced homomorphism of $R$-rings
$\varphi:\TGWC{\mu}{R}{\si}{t}\to A_E(\K)$. Suppose $a\in \TGWC{\mu}{R}{\si}{t}$ is a homogeneous element which lies in the radical of the gradation form $f$ on $\TGWC{\mu}{R}{\si}{t}$. Then in particular $0=f(a,a^\ast)=a\cdot a^\ast$. Hence $\varphi(a)\cdot\varphi(a)^\ast=0$ which implies that $\varphi(a)=0$, since $A_n$ is a domain.
Thus, by Theorem \ref{thm:nondeg}, $\varphi$ induces a homomorphism $\varphi_\Gamma:\TGWA{\mu}{R}{\si}{t}\to A_E(\K)$ which is the required map of $R$-rings with involution. 
\end{proof}

\begin{rem}
It is well-known that $R_E\to  A_E(\K)$, $u_e\mapsto y_ex_e$, is injective. Using this fact and that $\varphi_\Ga$ is a map of $R_E$-rings, it follows that the canonical map $\rho:R_E\to \mathcal{A}(\Gamma)$ is also injective. This gives another independent proof that $\mathcal{A}(\Gamma)$ is consistent and that the TGW datum $(R_E,\si^\Ga,t^\Ga)$ satisfies the consistency relations \eqref{eq:consistency_rels}. 
\end{rem}

\begin{example}
For the first four examples of Section \ref{sec:examples} we have:
\begin{enumerate}[{\rm (1)}]
\item $\varphi_\Ga$ is the identity map $\K\to \K$. (Since $E=\emptyset$, the Weyl algebra $A_E(\K)$ is just $\K$.)
\item $\varphi_\Ga$ is the evaluation homomorphism $\K[X_v,X_v^{-1}\mid v\in V]\to \K$ given by $\varphi_\Ga(X_v)=1$ for each $v\in V$.
\item $\varphi_\Ga$ is the embedding of $R_E$ into the Weyl algebra $A_E(\K)$ given by $\varphi_\Ga(u_e)=u_e=y_ex_e$ for each $e\in E$.
\item $\varphi_\Ga:\mathcal{A}(\Ga)\simeq A_I(\K)$ is an isomorphism mapping $X_i$ (respectively $Y_i$) to $x_i$ (respectively $y_i$) for all $i\in I$.
\end{enumerate}
\end{example}

\subsection{Faithfulness}
\label{sec:faithfulness}

In this section we give a precise condition, in terms of the graph $\Gamma$, for when the representation $\varphi_\Gamma:\mathcal{A}(\Gamma)\to A_E(\K)$ is faithful.
We need several definitions.
\begin{enumerate}[{\rm (i)}]
\item A \emph{subgraph} $\Gamma'=(V',E',s',t')$ of a  multiquiver $\Gamma=(V,E,s,t)$ is a multiquiver such that
$V'\subseteq V$, $E'\subseteq E$ and 
\[s'(e)=\begin{cases} s(e),&
\text{if $s(e)=(v,n)$ and $v\in V'$}\\
\emptyset,&\text{otherwise}
\end{cases}\qquad
t'(e)=\begin{cases} t(e),&
\text{if $t(e)=(v,n)$ and $v\in V'$}\\
\emptyset,&\text{otherwise}
\end{cases}
 \]
for all $e\in E'$.

\item A multiquiver $\Gamma$ is a \emph{directed cycle} if $V$ and $E$ are finite sets of the same cardinality and there are enumerations $V=\{v_1,v_2,\ldots,v_n\}$ and $E=\{e_1,e_2,\ldots,e_n\}$ such that for all $i\in\iv{1}{n}$ we have $s(e_i)=(v_i,n_i)$ and $t(e_i)=(v_{i+1},m_i)$ for some positive integers $n_i,m_i$ and we put $v_{n+1}=v_1$.

\item Two multiquivers $\Ga_1$ and $\Ga_2$ are called \emph{sign-equivalent} if they become isomorphic after changing the direction of some edges. In other words, the incidence matrices $\ga_1$ and $\ga_2$ will coincide (up to permutation of rows and columns) after multiplying some rows in $\ga_1$ or $\ga_2$ by $-1$.

\item A multiquiver $\Gamma=(V,E,s,t)$ is a \emph{not necessarily directed (NND) cycle} if it is sign-equivalent to a directed cycle.

\item If $\Ga'$ is a subgraph of a multiquiver $\Ga$, and $\Ga'$ is a (directed or NND) cycle, we simply say that $\Ga'$ is a (directed or NND) cycle \emph{in $\Ga$}.

\item A directed cycle $\Gamma$ is \emph{balanced} if the product of all outgoing multiplicities equals the product of all incoming multiplicities. This is equivalent to $\det(\ga)=0$, where $\ga$ is the incidence matrix of $\Ga$.

\item An NND cycle $\Ga$ is \emph{balanced} if it is sign-equivalent to a balanced directed cycle.

\item A connected component of a multiquiver $\Gamma$ is said to be \emph{in equilibrium} if it has finitely many vertices, no leaves and every NND cycle in $\Gamma$ is balanced.
\end{enumerate}

\begin{example}\label{ex:graph1}
This multiquiver has a disconnected vertex $v_1$, a disconnected leaf $e_1$, a connected leaf $e_5$, and one unbalanced cycle. Only the connected component consisting of $v_1$ is in equilibrium.
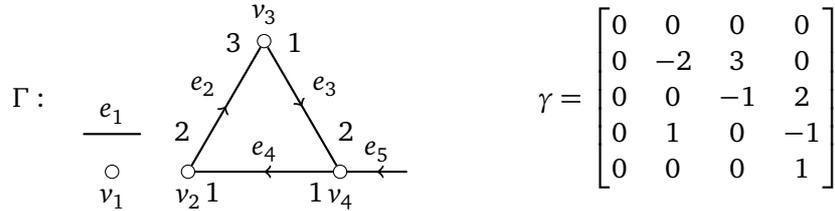
\[ 
\Ga:\;\;\begin{tikzpicture}[baseline=(60:1)]
\node[vertex] (A) at (-1,0) [label=below:$v_1$] {};
\node (1) at (-1.5,.5) {};
\node (2) at (-.5,.5) {};
\node (E) at (3,0) {};
\node[vertex] (B) at (0,0) [label=below:$v_2$] {};
\node[vertex] (C) at (60:2) [label=above:$v_3$] {};
\node[vertex] (D) at (2,0) [label=below:$v_4$] {};

\draw[disconnected leaf] (1) -- node[above] {$e_1$} (2);
\draw[leaf] (E) -- node[above] {$e_5$} (D);
\draw[edge] (B) -- node[above left] {$e_2$} node[auto,very near start] {$2$} node[auto,very near end] {$3$} (C);
\draw[xshift=10pt] (B) -- (C); 
\draw[edge] (C) -- node[above right] {$e_3$} node[auto,very near start] {$1$} node[auto,very near end] {$2$} (D);
\draw[edge] (D) -- node[above] {$e_4$} node[auto,very near start] {$1$} node[auto,very near end] {$1$} (B);
\end{tikzpicture}
\qquad
\qquad
\ga=\begin{bmatrix}
0 & 0 & 0 & 0 \\ 
0 &-2 & 3 & 0 \\
0 & 0 &-1 & 2 \\
0 & 1 & 0 &-1 \\
0 & 0 & 0 & 1
\end{bmatrix}
\] 
\end{example}

\begin{example}\label{ex:graph2}
This is a connected multiquiver in equilibrium since it is a balanced NND cycle.
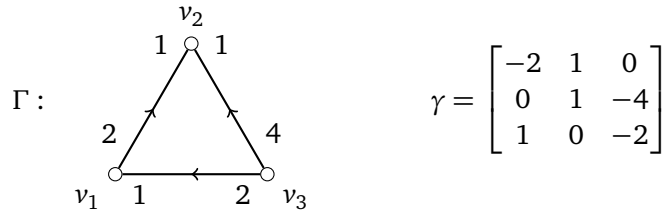
\[ 
\Ga:\;\;\begin{tikzpicture}[baseline=(60:1)]
\node[vertex] (A) at (0,0) [label=below left:$v_1$] {};
\node[vertex] (B) at (60:2) [label=above:$v_2$] {};
\node[vertex] (C) at (2,0) [label=below right:$v_3$] {};

\draw[edge] (A) -- node[auto,very near start] {$2$} node[auto,very near end] {$1$} (B);
\draw[edge] (C) -- node[above right,very near start] {$4$} node[above right,very near end] {$1$} (B);
\draw[edge] (C) -- node[auto,very near start] {$2$} node[auto,very near end] {$1$} (A);
\end{tikzpicture}
\qquad
\qquad
\ga=\begin{bmatrix}
-2& 1 & 0  \\ 
0 & 1 &-4  \\
1 & 0 &-2 
\end{bmatrix}
\] 
\end{example}

\begin{example}\label{ex:graph3}
This connected multiquiver without leaves is not in equilibrium since it has a non-balanced cycle.
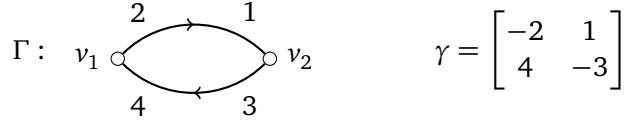
\[ 
\Ga:\;\;\begin{tikzpicture}[baseline=(A.base)]
\node[vertex] (A) at (0,0) [label=left:$v_1$] {};
\node[vertex] (B) at (2,0) [label=right:$v_2$] {};

\draw[edge] (A) to[out=45,in=135] node[auto,near start] {$2$} node[auto,near end] {$1$} (B);
\draw[edge] (B) to[out=-135,in=-45] node[auto,near start] {$3$} node[auto,near end] {$4$} (A);
\end{tikzpicture}
\qquad
\qquad
\ga=\begin{bmatrix}
-2& 1   \\ 
4 &-3   \\
\end{bmatrix}
\] 
\end{example}

The next result describes the centralizer of $R_E$ in $\mathcal{A}(\Gamma)$.
\begin{lem}\label{lem:centralizer}
Let $\Ga=(V,E,s,t)$ be a multiquiver, $\ga:\Z V\to \Z E$ its incidence matrix, $A=\mathcal{A}(\Ga)$ and $C_A(R_E)=\{a\in A\mid ar=ra\,\forall r\in R_E\}$ be the centralizer of $R_E$ in $A$. Then 
\[C_A(R_E)=\bigoplus_{d\in\ker \ga} A_d\]
In particular, $\ga$ is injective if and only if $R_E$ is maximal commutative in $A$.
\end{lem}
\begin{proof}
Let $K$ be the kernel of the group homomorphism $\hat\sigma:\Z V\to \Aut_\K(R_E)$ determined by $\hat\sigma(v)=\sigma_v$ for all $v\in V$. By definition \eqref{eq:sigma-definition} of $\si_v$,  we have
\[\hat\sigma(d)(u_e)=u_e-\ga(d)_e\]
where $\ga(d)=\sum_{e\in E}\ga(d)_e e$.
Thus $K=\ker \ga$. By \cite[Thm.~5.1]{HarOin2013} the claim follows. 
\end{proof}

The following theorem describes the rank of $\ker \ga$ as a free abelian group, in terms of the graph $\Ga$. In view of Section \ref{sec:symmetric}, it is parallel to \cite[Thm.~5.8]{HarOin2013}, since if $\Ga$ is a symmetric simple quiver, then every connected component is in equilibrium.
\begin{thm}\label{thm:rank-ker-gamma}
Let $\Gamma$ be a multiquiver and $\gamma$ be its incidence matrix. Then the rank of the kernel of $\gamma$ is equal to the number of connected components of $\Gamma$ in equilibrium.
\end{thm}
\begin{proof}
Put $K=\ker \ga$. Let $\mathcal{C}(\Ga)$ be the set of connected components of $\Ga$, and $\mathcal{C}^{\mathrm{eq}}(\Ga)\subseteq\mathcal{C}(\Ga)$ be the subset of components in equilibrium. We have
\[\Z V=\bigoplus_{C\in\mathcal{C}(\Ga)}\Z V_C\]
where $V_C$ is the vertex set of $C$,
and $\ga$ is block diagonal with respect to this decomposition (since no two vertices belonging to different components are adjacent). Therefore
\[K=\bigoplus_{C\in\mathcal{C}(\Ga)} K_C\]
where $K_C=K\cap \Z V_C$.
Furthermore, the rank of each subgroup $K_C$ is at most one.
Indeed, assume $K_C\neq\{0\}$ and let $a\in K_C\setminus\{0\}$. Write $a=\sum_{v\in V_C} \lambda_v v$, where $\lambda_v\in \Z$.
Let $e$ be any proper edge in $C$, connecting two vertices $v,w\in V_C$. Since $a\in K$, in particular the coefficient of $e$ in $\ga(a)$ is zero. That is, $0=\ga(a)_e=\la_v \ga_{ev} + \la_w \ga_{ew}$.
Since $C$ is connected, it follows that all coefficients $\la_v$ are nonzero and uniquely determined by any single one of them. Thus the rank of $K_C$ equals one.

Suppose $C\in\mathcal{C}(\Ga)$ is not in equilibrium.
Let $a=\sum_{v\in V_C}\la_v v\in K_C$. 
If $V_C$ is infinite, then $\la_v=0$ for some $v\in V_C$, hence $\la_w=0$ for all $w\in V_C$ since $C$ is connected, hence $a=0$.
If $C$ has a connected leaf $e$, let $v\in V_C$ be the vertex incident to $e$. Then $0=\ga(a)_e=\la_v\ga_{ev}$ which implies $\la_v=0$. As above that implies that $a=0$.
If $C$ has a non-balanced cycle $\Ga'=(V',E',s',t')$, let $\ga':\Z V'\to \Z E'$ be the incidence matrix of $\Ga'$. Since $\Ga'$ is non-balanced, $\det(\ga')\neq 0$. Let $a'=\sum_{v\in V'} \la_v v$. Since $\ga(a)=0$ we get $\ga'(a')=0$ which implies $a'=0$. So $\la_v=0$ for all $v\in V'$. Since $C$ is connected, $\la_v=0$ for all $v\in V$. This shows that $K_C=\{0\}$ for $C$ not in equilibrium.

Suppose $C\in\mathcal{C}(\Ga)$ is a component that is in equilibrium. We have to show that $K_C$ is nonzero.
Fix any vertex $v_0$ in $C$ and define $\la_{v_0}=1$. For any other vertex $w$ in $C$, pick a path (linear subgraph) $p$ from $v_0$ to $w$. Let $\vec{p}$ be the multiquiver which is sign-equivalent $p$ but every edge is directed forward (from $v_0$ to $w$).
Define $\la_w=O_{\vec{p}}/I_{\vec{p}}$
where $O_{\vec{p}}$ (respectively $I_{\vec{p}}$) is the product of all the outgoing
(respectively incoming) multiplicities in edges in $\vec{p}$. This is independent of the choice of path since all cycles are balanced. Put
$a'=\sum_{v\in V_C} \la_v v \in \Q V_C$. Let $k$ be an integer such that $a=ka'\in \Z V_C$. 
We claim that $a\in K_C\setminus\{0\}$. Clearly $a\neq 0$.
Let $e$ be any edge in $C$, between $w_1$ and $w_2$. Without loss of generality we may assume there is a path from $v_0$ to $w_2$ that goes through $w_1$. Then $\la_{w_2}=\la_{w_1}\cdot \frac{m_1}{m_2}$ where $m_i$ is the multiplicity of $e$ at $w_i$. Thus
\[\ga(a)_e=k\la_{w_1}\ga_{e,w_1}+k\la_{w_2}\ga_{e,w_2}=k\la_{w_1}(\ga_{e,w_1}+(m_1/m_2)\ga_{e,w_2})=\pm k\la_{w_1}(m_1-(m_1/m_2)m_2)=0.\]
Since $e$ was arbitrary, this proves that $a\in K_C$.
\end{proof}

\begin{cor}
\label{cor:centralizer}
Let $\Ga=(V,E,s,t)$ be any multiquiver and $A=\mathcal{A}(\Ga)$ the associated TGW algebra. Then the centralizer $C_A(R_E)$ is a maximal commutative subalgebra of $A$.
\end{cor}
\begin{proof} From the relations \eqref{eq:tgwarels} one checks that $\si_j(t_j)X_iX_j=\si_j\si_i(t_j)X_jX_i$ for all $i,j\in V$, $i\neq j$. Thus, since $A$ is a domain, if $\si_i(t_j)=t_j$ it follows that $X_i$ and $X_j$ commute. Similarly (or by applying the involution), $Y_i$ and $Y_j$ commute if $\si_i(t_j)=t_j$. This implies that if $C$ and $C'$ are two connected components of $\Ga$, then $[A_g,A_h]=0$ for any $g\in \Z V_C$, $h\in \Z V_{C'}$. Thus, by \cite[Thm.~5.3]{HarOin2013} and Lemma \ref{lem:centralizer}, it follows that $C_A(R_E)$ is commutative. Therefore, since $R_E$ is commutative, $C_A(R_E)$ is maximal commutative in $A$.
\end{proof}

\begin{rem} \label{rem:markov}
We would like to mention a connection with Markov chains.
Assume that $\Ga=(V,E,s,t)$ is a connected simple quiver. 
As in Remark \ref{rem:simple-quiver}, let $Q(\Ga)$ be the digraph corresponding to $\Ga$. An edge $e\in E$ from $i$ to $j$ yields two edges in $Q(\Ga)$, the one from $i$ to $j$ is labeled by $v_{ij}=|\ga_{ei}|$ and the one from $j$ to $i$ is labeled by $v_{ji}=|\ga_{ej}|$.
If there is no edge between $i$ and $j$ in $\Ga$, we put $v_{ij}=v_{ji}=0$.
Then $d=\sum_{v\in V} d_v v\in\Z V$ is in $\ker \ga$ if and only if $d_iv_{ij}=d_jv_{ji}$ for all $i,j\in V$. That is, $\ker \ga$ is nonzero if and only if the matrix $(v_{ij})_{i,j\in V}$ is \emph{symmetrizable}.

For each $i\in V$, put $Z_i=\sum_{j\in V}v_{ij}$ and define $p_{ij}=v_{ij}/Z_i$. Then one may view $Q(\Ga)$ as a discrete-time Markov chain with state space $V$ and transition matrix $P=(p_{ij})_{i,j\in V}$. Thus, at each time step, $p_{ij}$ is the probability that the system will jump from state $i$ to state $j$. Let $\{q_i\}_{i\in V}$ denote the stationary distribution.
The system is \emph{reversible} if and only if $q_ip_{ij}=q_jp_{ji}$ for all $i,j$, which holds if and only if $q=\sum_{v\in V}(q_v/Z_v) v$ is in $\ker \ga_\mathbb{Q}$ where $\ga_{\mathbb{Q}}:\mathbb{Q}V\to\mathbb{Q}E$ is the rational extension of $\ga$. Thus, by Theorem \ref{thm:rank-ker-gamma}, the system is reversible if and only if $\Ga$ is in equilibrium. The latter means exactly that
$p_{i_1i_2}p_{i_2i_3}\cdots p_{i_ki_1}=p_{i_ki_{k-1}}\cdots p_{i_3i_2}p_{i_2i_1}$ for any directed cycle in $Q(\Ga)$ with vertices $i_1,i_2,\ldots,i_k$, which is known as \emph{Kolmogorov's criterion for detailed balance}.
\end{rem}

We can now prove the main theorem of this section.
\begin{thm}\label{thm:faithfulness}
Let $\Gamma$ be a multiquiver, $\gamma$ be its incidence matrix, $\mathcal{A}(\Ga)$ the corresponding TGW algebra, and $\varphi_\Gamma$ be the canonical representation by differential operators \eqref{eq:varphi_def}.
Then the following statements are equivalent:
\begin{enumerate}[{\rm (i)}]
\item $\varphi_\Gamma$ is injective;
\item $\gamma$ is injective;
\item $R_E$ is a maximal commutative subalgebra of $\mathcal{A}(\Gamma)$;
\item No connected component of $\Ga$ is in equilibrium.
\end{enumerate}
\end{thm}
\begin{proof}
(i)$\Rightarrow$(ii):
Assume that $\gamma(d_1)=\gamma(d_2)$ for some $d_1,d_2\in\Z V$. Pick nonzero elements $a_i\in\mathcal{A}(\Gamma)_{d_i}$ and let $b_i=\varphi_\Gamma(a_i)$ for $i=1,2$. From the definition of $\varphi_\Gamma$ we get $\deg (b_1)=\gamma(d_1)=\gamma(d_2)=\deg(b_2)$. Since each homogeneous component of the Weyl algebra $A_E(\K)$ is a free cyclic $R_E$-module there are nonzero $r_1,r_2\in R_E$ such that $r_1b_1=r_2b_2$. Since $\varphi_\Gamma$ is a map of $R_E$-rings we get $\varphi_\Ga(r_1a_1)=\varphi_\Ga(r_2a_2)$. If $\varphi_\Ga$ is injective this implies $r_1a_1=r_2a_2$, which by Corollary \ref{cor:domain} implies $d_1=d_2$.

(ii)$\Rightarrow$(iii): Follows by Lemma \ref{lem:centralizer}.

(iii)$\Rightarrow$(i):
Suppose $R_E$ is a maximal commutative subalgebra of $\mathcal{A}(\Gamma)$.
That means that $C_A(R_E)=R_E$. Since $\varphi_\Gamma$ is a map of $R_E$-rings, $R_E\cap\ker(\varphi_\Gamma)=0$, it follows by  
\cite[Thm.~3.6]{HarOin2013}
that $\ker(\varphi_\Gamma)=0$.

(ii)$\Leftrightarrow$(iv): Immediate by Theorem \ref{thm:rank-ker-gamma}.

\end{proof}

\begin{example}
$\varphi_\Ga$ is faithful in Example \ref{ex:graph3}, but not in Examples \ref{ex:graph1}-\ref{ex:graph2}.
\end{example}

\subsection{Local surjectivity}\label{sec:local-surjectivity}
Let $\Gamma$ be a multiquiver, $\ga$ its incidence matrix and $\varphi_\Ga$ the representation from Section \ref{sec:homomorphism-varphi}.
We say that $\varphi_\Gamma$ is \emph{locally surjective} if 
\begin{equation}
\varphi_\Ga(\mathcal{A}(\Ga)_g)=A_E(\K)_{\ga(g)},\qquad\forall g\in \Z V.
\end{equation}
Note that $\gamma$ (hence $\Gamma$) is recoverable from $\varphi_\Gamma$ since $\gamma(v)\in \Z E$ is the degree of $\varphi_\Gamma(X_v)$ for all $v\in V$.

In this section we show that the presence of (not necessarily directed) cycles in $\Gamma$ is precisely the obstruction to $\Phi_\Gamma$ being locally surjective.
\begin{rem}
The notion of local surjectivity is natural in the following sense.
A group-graded algebra $A=\bigoplus_{g\in G} A_g$ can be viewed as a category with object set $G$ and morphism sets $\Hom(g,h)=A_{hg^{-1}}$. In particular, if $\Gamma=(V,E,s,t)$ is a multiquiver, then $\mathcal{A}(\Gamma)$ is $\Z V$-graded and the Weyl algebra $A_E(\K)$ is $\Z E$-graded and we may regard them as categories in this way. The two maps $\gamma\in\Hom_\Z(\Z V, \Z E)$ and $\varphi_\Gamma:\mathcal{A}(\Gamma)\to A_E(\K)$ are naturally the object and morphism components of a functor $\Phi_\Gamma$ from $\mathcal{A}(\Ga)$ to $A_E(\K)$, viewed as categories. Then $\varphi_\Ga$ is locally surjective if and only if $\Phi_\Gamma$ is a \emph{full} functor.
\end{rem}

\begin{lem}\label{lem:alpha_lemma}
Let $E,V$ be sets and $\ga=(\ga_{ev})_{e\in E,v\in V}$ be an integer $E\times V$-matrix satisfying
conditions \eqref{eq:condition_C}-\eqref{eq:condition_M}.
Let $g=\sum_{v\in V} g_vv$ and $h=\sum_{v\in V} h_vv$ be elements of $\Z V$ satisfying
\begin{equation}\label{eq:alpha_lemma_1}
\begin{aligned}
g_ih_i &\ge 0,\quad\forall i\in V,\\
g_ih_j &\le 0,\quad\forall i,j\in V, i\neq j.
\end{aligned}
\end{equation}
Then 
\begin{equation}
A_E(\K)_{\ga(g)}A_E(\K)_{\ga(h)}=A_E(\K)_{\ga(g+h)}.
\end{equation}
\end{lem}
\begin{proof}
Since $\ga$ satisfies \eqref{eq:condition_M} we have
\begin{equation}\label{eq:pf_alpha_lemma}
\ga_{ei}\ga_{ej}\le 0,\qquad\forall e\in E, i,j\in V, i\neq j. 
\end{equation}
Using \eqref{eq:alpha_lemma_1} and \eqref{eq:pf_alpha_lemma} we have
\[\big(\sum_{i\in V} \ga_{ei}g_i\big) \big(\sum_{i\in V} \ga_{ei}h_i\big)\ge 0.
\qquad\forall e\in E,\]
Hence, putting $\ga(d)_e=\sum_{i\in V} \ga_{ei}d_i$ for $d=\sum_{v\in V} d_v v\in\Z V$, we have
\[z_e^{(\ga(g)_e)}z_e^{(\ga(h)_e)}=z_e^{(\ga(g+h)_e)}.\]
Taking the product over $e\in E$ (at most finitely many factors being $\neq 1$ by \eqref{eq:condition_C}) we get
\[z^{\ga(g)}z^{\ga(h)} = z^{\ga(g+h)}.\]
Since $A_E(\K)_p$ is a free cyclic as a left $R_E$-module with generator $z^{p}$ for each $p\in \Z E$, the claim follows.
\end{proof}

Let $\Gamma=(V,E,s,t)$ be a multiquiver. The \emph{underlying undirected graph}, denoted $\bar \Ga$, is obtained by forgetting the direction of each edge. Formally, $\bar\Ga$ can be identified with the sign-equivalence class of multiquivers containing $\Ga$. We will prove a lemma about cycles in $\bar\Ga$.
Let $L\subseteq E$ be the set of leaves and put $E^\text{prop}=E\setminus L$. Thus $E^\text{prop}=\{e\in E\mid s(e)\neq\emptyset\neq t(e)\}$. The edges in $E^\text{prop}$ will be called \emph{proper}.
To each total order $<$ on $V$ we associate a parity function
\begin{equation}\label{eq:parity-def}
\mathcal{P}_<:E^\text{prop} \to \{-1,1\},\qquad 
\mathcal{P}_<(e)=
\begin{cases}
1,&s_1(e)<t_1(e),\\ 
-1,&t_1(e)<s_1(e),
\end{cases}
\end{equation}
where $s_1(e)=v$ if $s(e)=(v,n)$ for some $n\in\Z_{>0}$ and similarly for $t_1$.
Equivalently,
\begin{equation}\label{eq:parity-alt}
\mathcal{P}_<(e)=
\begin{cases}
1,& \ga_{ev}<0, \ga_{ew}>0 \\ 
-1, & \ga_{ev}>0, \ga_{ew}<0
\end{cases}
\quad \text{where $\{v,w\}=\{s_1(e),t_1(e)\}$, $v<w$.}
\end{equation}

\begin{lem}\label{lem:acyclic_equiv}
Let $\Ga=(V,E,s,t)$ be a multiquiver with $V$ finite and $\bar\Ga$ be the underlying undirected graph. The following statements are equivalent:
\begin{enumerate}[{\rm (i)}]
\item $\bar\Gamma$ is acyclic;
\item For any function $f:E^\text{prop}\to\{-1,1\}$ there exists a total order $<$ on $V$ such that $f=\mathcal{P}_<$.
\end{enumerate} 
\end{lem}
\begin{proof}
(i)$\Rightarrow$(ii): 
We can assume that $\Ga$ is connected. Consider $\Ga$ as a rooted tree by picking any vertex $j_0\in V$ as the root. We think of a total order on $V$ as a visual representation of $\Ga$ where $v<w$ means $v$ is placed to the left of $w$.
Any desired parity of the edges connected to the root vertex $j_0$ can be obtained by permuting the child vertices to the left or right of $j_0$.
Then continue inductively in this fashion with each child of the root. Since $V$ is finite this process will eventually end.

(ii)$\Rightarrow$(i): Suppose $\Ga$ has a not necessarily directed cycle. Without loss of generality we can assume this cycle is all of $\Ga$.
Let $n=\# V$ and denote the elements of $V$ by $v_1,v_2,\ldots,v_n$ such that there is an edge $e_{ij}$ between  $v_i$ and $v_j$ if and only if $j\equiv i+1 \text{(mod $n$)}$. Let $<$ be the total order on $V$ given by $v_1<v_2<\cdots <v_n$. 
Let $f:E\to \{-1,1\}$ be the function that differs from $\mathcal{P}_<$ only at the edge $e_{n1}$.
By (ii) there is a total order $\prec$ on $V$ such that $f=\mathcal{P}_\prec$.
Then $\mathcal{P}_\prec(e_{i,i+1})=f(e_{i,i+1})=\mathcal{P}_<(e_{i,i+1})$ which implies that $v_i\prec v_{i+1}$ for each $i\in\{1,2,\ldots,n-1\}$. Hence $v_1\prec v_n$ by transitivity, which implies that $f(e_{n1})=\mathcal{P}_\prec(e_{n1})=\mathcal{P}_<(e_{n1})=-f(e_{n1})$ which is a contradiction.
\end{proof}

The following lemma introduces certain useful polynomials $P_{mn}(u)$.
\begin{lem}\label{lem:Pmn}
For all integers $m,n\in \Z$ there exists a unique polynomial $P_{mn}(u)$ such that the following identity holds in the first Weyl algebra $A_1(\K)=\K\langle x,y\mid yx-xy=1\rangle$:
\begin{equation}
z^{(m)}z^{(n)}=P_{mn}(u) z^{(m+n)}
\end{equation}
where $u=yx$ and we use notation \eqref{eq:zk_notation}.
If $mn\ge 0$ then $P_{mn}(u)=1$ (constant polynomial).
If $mn<0$ then $P_{mn}(u)$ is a product of factors of the form $(u-k)$, $k\in\Z$.
If $m>0, n<0$, then all zeroes of $P_{mn}(u)$ are positive integers.
If $m<0, n>0$, then all zeroes of $P_{mn}(u)$ are non-positive integers.
\end{lem}
\begin{proof} Follows by direct calculation using the presentation of $A_1(\K)$ as a generalized Weyl algebra. For example, $x^4y^2=x^3(u-1)y=(u-4)x^3y=(u-4)x^2(u-1)=(u-4)(u-3)x^2$ so $P_{4,-2}(u)=(u-4)(u-3)$.
\end{proof}

Next we calculate some values of $\varphi_\Ga$ in terms of the polynomials $P_{mn}(u)$.
If $<$ is a total order on a finite set $I$ and $x_i$ for $i\in I$ are not necessarily commuting variables, we put $\prod_{i\in I}^{<} x_i=x_{i_1}x_{i_2}\cdots x_{i_k}$ if $I=\{i_1,i_2,\ldots,i_k\}$ and $i_1<i_2<\cdots<i_k$.
\begin{lem} \label{lem:varphi-calculation}
Let $\Ga=(V,E,s,t)$ be a multiquiver, $<$ be a total order on $V$, and $V'$ be a finite subset of $V$. Then 
\begin{equation}\label{eq:varphi-calculation}
\varphi_\Ga(\prod_{v\in V'}^{<} X_v) = \Big(\prod_{\substack{v,w\in V'\\ v<w}} \prod_{\substack{e\in E\\ \{s_1(e),t_1(e)\}=\{v,w\}}} P_{\ga_{ev},\ga_{ew}}(u_e) \Big) z^{\ga(g)}
\end{equation}
where $g=\sum_{v\in V'} v\in \Z V$, $s_1(e)=v$ if $s(e)=(v,n)$ for some $n\in\Z_{>0}$ and similarly for $t_1$,  $P_{mn}(u)$ are the polynomials from Lemma \ref{lem:Pmn}, and $\ga_{ev}$ are the entries of the incidence matrix $\ga$.
\end{lem}
\begin{proof} By definition of $\varphi_{\Ga}$ we have
\[ 
\varphi_\Ga(\prod_{v\in V'}^{<} X_v) =\prod_{v\in V'}^{<} z^{\ga(v)}   
=\prod_{v\in V'}^{<}\prod_{e\in E} z_e^{(\ga_{ev})} \] 
Since $[x_e,x_{e'}]=[x_e,y_{e'}]=[y_e,y_{e'}]=0$ for $e,e'\in E$, $e\neq e'$ we may switch the order of the products to get
\begin{equation}\label{eq:varphi-calc-step}
\prod_{e\in E} \prod_{v\in V'}^{<} z_e^{(\ga_{ev})}.
\end{equation}
If $e$ is an edge between two vertices $v,v'$ in $V'$ with $v<v'$ we get
\[\prod_{v\in V'}^{<} z_e^{(\ga_{ev})} = z_e^{(\ga_{ev})}z_e^{(\ga_{ev'})} = P_{\ga_{ev},\ga_{ev'}}(u_e) z^{(\ga_{ev}+\ga_{ev'})}\] 
by Lemma \ref{lem:Pmn}. For the other edges there are at most one $v\in V'$ such that $\ga_{ev}\neq 0$.
Thus \eqref{eq:varphi-calc-step} equals
\[\Big( \prod_{\substack{v,w\in V'\\ v<w}}\prod_{\substack{e\in E\\ \{s_1(e),t_1(e)\}=\{v,w\}}}
P_{\ga_{ev},\ga_{ew}}(u_e) \Big)\cdot  \prod_{e\in E} z_e^{\big(\sum_{v\in V'}\ga_{ev}\big)} \] 
which finishes the proof.

\end{proof}

We now prove the main result in this section.

\begin{thm}\label{thm:local-surjectivity}
Let $\Gamma=(V,E,s,t)$ be a multiquiver with incidence matrix $\gamma$ and let $\mathcal{A}(\Gamma)$ be the corresponding TGW algebra. Let $\varphi_\Gamma:\mathcal{A}(\Gamma)\to A_E(\K)$ be the representation from Theorem \ref{thm:map_to_Weyl_algebra}. Let $\bar\Ga$ denote the underlying undirected graph of $\Ga$.
Then the following two statements are equivalent:
\begin{enumerate}[{\rm (i)}]
\item $\varphi_\Ga$ is locally surjective;
\item $\bar\Gamma$ is acyclic.
\end{enumerate}
\end{thm}

\begin{proof}
(ii)$\Rightarrow$(i): Assume $\bar\Gamma$ has no cycles. Put $\varphi=\varphi_\Gamma$. We have to prove that
\begin{equation}\label{eq:pfb_1}
\varphi\big(\mathcal{A}(\Gamma)_g\big)=A_E(\K)_{\ga(g)},
\end{equation}
for all $g=\sum_{v\in V} g_v v\in\Z V$.
We first reduce to the case when $g_v\ge 0$ for all $v\in V$.
Decompose $g=g_++g_-$ where $g_+=\sum_{v\in V} (g_+)_v v$ $(g_+)_v=\max(0,g_v)$ and $g_-=g-g_+$. By Lemma \ref{lem:monomials},
\begin{equation}\label{eq:pfb_3}
\varphi(\TGWA{\mu}{R}{\si}{t}_g)=\varphi(\TGWA{\mu}{R}{\si}{t}_{g_-})\varphi(\TGWA{\mu}{R}{\si}{t}_{g_+}).
\end{equation}
Assuming \eqref{eq:pfb_1} holds for $g_+$ and $-g_-$ and using that $\varphi$ is a map of rings with involution, \eqref{eq:pfb_3} and Lemma \ref{lem:alpha_lemma} imply that
\begin{equation}\label{eq:pfb_4}
\varphi(\TGWA{\mu}{R}{\si}{t}_g)=A_E(\K)_{\ga(g_-)}A_E(\K)_{\ga(g_+)}=A_E(\K)_{\ga(g)}.
\end{equation}
Thus we may assume that $g_i\ge 0$ for all $i\in V$.

Next we reduce to the case when $V$ is a finite set and $g_i>0$ for each $i\in V$ by the following argument. Let $V'=\{i\in V\mid g_i\neq 0\}$ and $\gamma'=(\ga_{ev})_{e\in E, v\in V'}$. Clearly $V'$ is a finite set. Let $\Gamma'$ be the multiquiver corresponding to $\gamma'$. In other words, $\Gamma'$ is obtained from $\Gamma$ by removing all vertices $v$ for which $g_v=0$ and turning edges at such $v$, if any, into (possibly disconnected) leaves.  The multiquiver $\Ga'$ has no NND cycles either.
Let $\mathcal{A}(\Gamma')$ be the corresponding TGW algebra with homomorphism $\varphi'=\varphi_{\Gamma'}$.
There is a commutative triangle
\begin{equation}
\begin{aligned}
\begin{tikzcd}[ampersand replacement=\&,
               column sep=small]
\mathcal{A}(\Gamma')
 \arrow{rr}{\psi}
 \arrow{dr}[swap]{\varphi'} 
 \& \& \mathcal{A}(\Gamma)
  \arrow{dl}{\varphi} \\ 
\& A_E(\K) \& 
\end{tikzcd}
\end{aligned}
\end{equation}
where $\psi$ is the inclusion map $\psi(X_v)=X_v, \psi(Y_v)=Y_v$ for all $v\in V'$, and $\psi(r)=r$ for all $r\in R_E$. 
Note that $g=\sum_{v\in V} g_v v=\sum_{v\in V'} g_v v\in \Z V'\subseteq \Z V$.
We have $\varphi(\mathcal{A}(\Ga)_g)=(\varphi\psi)(\mathcal{A}(\Ga')_g)=\varphi'(\mathcal{A}(\Ga')_g)$. 
Thus we may without loss of generality assume that $V$ is finite, and $g_i>0$ for all $i$.
For notational purposes we may assume that $V=\{1,2,\ldots,m\}$.

Finally, we reduce to the case when $g_i=1$ for all $i\in V$. Since $A_E(\K)_{\ga(g)}=R_E z^{\ga(g)}$, relation \eqref{eq:pfb_1} follows if we prove that
\begin{equation}\label{eq:pfb_5}
z^{\ga(g)}\in \sum_{\pi\in S_m} R_E \varphi(X_{\pi(1)}^{g_{\pi(1)}}\cdots X_{\pi(m)}^{g_{\pi(m)}})
\end{equation}
where $S_m$ denotes the symmetric group.
If some $g_i>1$ we may replace $g_i$ by $1$ and the $i$:th column of $\ga$ by
$(g_i\ga_{ei})_{e\in E}$.
The new matrix $\ga_g$ still satisfies condition \eqref{eq:condition_M} and the identity \eqref{eq:pfb_5} for $\ga_g$ will coincide with the original one.
Therefore we may assume that $g_i=1$ for all $i\in V$.

By Lemma \ref{lem:varphi-calculation}, we have
\begin{equation}\label{eq:pfb_7}
\varphi(X_{\pi(1)}\cdots X_{\pi(m)}) =
\prod_{i<j}\prod_{e\in E} P_{\ga_{e\pi(i)},\ga_{e\pi(j)}}(u_e) z^{\ga(g)}.
\end{equation}
If $\pi(i)$ or $\pi(j)$ is not incident to $e$, then $P_{\ga_{e\pi(i)},\ga_{e\pi(j)}}=1$. If $e$ is a proper edge between $\pi(i)$ and $\pi(j)$, $i<j$, then by \eqref{eq:parity-alt} and Lemma \ref{lem:Pmn},
\[
\mathcal{V}(P_{\ga_{e\pi(i)},\ga_{e\pi(j)}})\subseteq 
\begin{cases}
\Z_{\le 0},&\mathcal{P}_\pi(e) = 1 \\ 
\Z_{>0},&\mathcal{P}_\pi(e) = -1 
\end{cases}
\]
where $\mathcal{V}(P)$ denote the set of zeroes in the algebraic closure $\bar\K$ of a polynomial $P\in \K[u]$, and we identified $\pi\in S_m$ with the total order $\prec$ on $V$ given by $\pi(i)\prec \pi(j)$ for all $i<j$.
Therefore, by (ii) and Lemma \ref{lem:acyclic_equiv},
the polynomials $\big\{\prod_{i<j}\prod_{e\in E} P_{\ga_{e\pi(i)},\ga_{e\pi(j)}}(u_e)\big\}_{\pi\in S_m}$ 
have no common zeroes in $\bar\K^E$.
By the weak Nullstellensatz, the ideal these polynomials generate in $R_E$ contains $1$. Hence
\[z^{\ga(g)}\in \sum_{\pi\in S_m} R_E \varphi_{\Ga}(X_{\pi(1)}\cdots X_{\pi(m)}).\] 
This finishes the proof that $\varphi_\Gamma$ is locally surjective.

(i)$\Rightarrow$(ii): 
Assume that $\Gamma$ has a not necessarily directed cycle $\Ga'=(V',E',s',t')$.
We will prove that $\varphi_\Ga$ is not locally surjective by showing that
for $g=\sum_{v\in V'} v$ we have
\[\varphi_\Ga(\mathcal{A}(\Ga)_g)\subseteq J \cdot A_E(\K)_{\ga(g)}\]
where $J$ is a proper ideal of $R_E$.
By Lemma \ref{lem:monomials},
\[\mathcal{A}(\Ga)_g=\sum_{\prec} R_E \cdot \prod_{v\in V'}^\prec X_v \] 
where we sum over all total orders $\prec$ on $V'$.
Thus, by Lemma \ref{lem:varphi-calculation},
\[\varphi_\Ga(\mathcal{A}(\Ga)_g)=\sum_{\prec} R_E\cdot
\Big( \prod_{\substack{v,w\in V'\\ v<w}}\prod_{\substack{e\in E\\ \{s_1(e),t_1(e)\}=\{v,w\}}} 
  P_{\ga_{ev},\ga_{ew}}(u_e)\Big) z^{\ga(g)}\] 
Since $P_{\ga_{ev},\ga_{ew}}=1$ unless $v$ and $w$ are both incident to $E$, we may without loss of generality assume that $\Ga'=\Ga$. In particular every edge is proper, and $V$ and $E$ are finite.
We have
\begin{equation}\label{eq:varphi-proof-identity}
\prod_{\substack{v,w\in V\\ v<w}}\prod_{\substack{e\in E\\ \{s_1(e),t_1(e)\}=\{v,w\}}} 
  P_{\ga_{ev},\ga_{ew}}(u_e) = 
  \prod_{\substack{e\in E\\ \mathcal{P}_<(e)=1}} P_e^+
  \prod_{\substack{e\in E\\ \mathcal{P}_<(e)=-1}} P_e^-  
\end{equation} 
where 
\[P_e^+ = P_{\ga_{e,s_1(e)},\ga_{e,t_1(e)}}(u_e),\qquad  
  P_e^- = P_{\ga_{e,t_1(e)},\ga_{e,s_1(e)}}(u_e). \] 
By Lemma \ref{lem:acyclic_equiv}, there exists a function $f:E\to \{-1,1\}$ such that $f\neq \mathcal{P}_\prec$ for all total orders $\prec$ on $V$. This means that for any total order $\prec$ on $V$, there exists $e_{\prec}\in E$ such that $\mathcal{P}_\prec (e_\prec ) = -f(e_{\prec})$.
Define
\begin{equation}\label{eq:ideal-J}
J=\sum_{\prec} R_E\cdot P_\prec 
\end{equation}
where 
\[
P_\prec = 
\begin{cases}
P_{e_\prec }^+, & \text{if $\mathcal{P}_\prec (e_\prec ) = 1 $} \\ 
P_{e_\prec }^-, & \text{if $\mathcal{P}_\prec (e_\prec ) = - 1$}
\end{cases}
\] 
We claim that $J$ is a proper ideal of $R_E$.
Assume that $\prec_1$ and $\prec_2$ are total orders on $V$ such that $e_{\prec_1}=e_{\prec_2}=e$.
Then $\mathcal{P}_{\prec_1}(e)=-f(e)=\mathcal{P}_{\prec_2}(e)$ and hence $P_{\prec_1}=P_{\prec_2}$.
Therefore $J$ is proper.
\end{proof}

\begin{example}
$\varphi_\Ga$ is locally surjective in
Examples (6)-(7) of Section \ref{sec:examples}, but not in Examples \ref{ex:graph1}-\ref{ex:graph3}.
\end{example}

\begin{example} We illustrate the proof of the theorem by considering a more detailed example.
Let $V=\{1,2,3\}$, $E=\{a,b,c\}$ and let $\Ga=(V,E,s,t)$ be the following multiquiver:
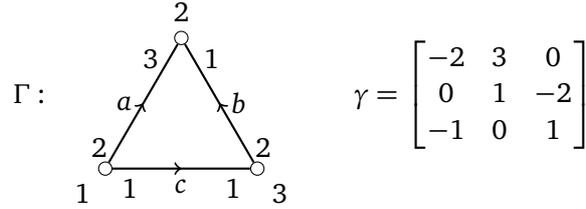
\[
\Ga:\;\;\begin{tikzpicture}[baseline=(60:1)]
\node[vertex] (A) at (0,0) [label=below left:$1$] {};
\node[vertex] (B) at (60:2) [label=above:$2$] {};
\node[vertex] (C) at (0:2) [label=below right:$3$] {};

\draw[edge] (A) -- node[left,very near start] {$2$} node[left] {$a$}  node[left,very near end] {$3$} (B);
\draw[edge] (C) -- node[right,very near start] {$2$} node[right] {$b$} node[right,very near end] {$1$} (B);
\draw[edge] (A) -- node[below,very near start] {$1$} node[below] {$c$}  node[below,very near end] {$1$} (C);
\end{tikzpicture}
\qquad 
\ga=\begin{bmatrix}
-2 & 3 & 0  \\ 
 0 & 1 &-2  \\ 
-1 & 0 & 1   
\end{bmatrix}
\] 
Let $\mathcal{A}(\Ga)$ be the TGW algebra associated to $\Ga$, and $\varphi_\Ga:\mathcal{A}(\Ga)\to A_E(\K)$ be the natural representation by differential operators. 
Since $\bar\Ga$ has a cycle, we expect $\varphi_\Ga$ to not be locally surjective. We check it using the methods and notation of the proof of Theorem \ref{thm:local-surjectivity}.
By definition of $\varphi_\Ga$ and relations in the Weyl algebra $A_E(\K)$ we have
\begin{align*}
\varphi_\Ga(X_1X_2X_3)&=y_a^2y_c\cdot x_a^3x_b\cdot y_b^2x_c \\ 
&= y_a^2x_a^3\cdot x_by_b^2\cdot y_cx_c \\ 
&= u_a(u_a+1)(u_b-1)u_c\cdot x_ay_b
\end{align*}
Recall the definition \eqref{eq:parity-def} of the parity function 
$\mathcal{P}_{\prec}$ attached to a total order $\prec$ on $V$.
Below we abbreviate the total order $i\prec j\prec k$ by $ijk$ and write
$\mathcal{P}_{ijk}(a,b,c)$ for $\big(\mathcal{P}_{ijk}(a),\mathcal{P}_{ijk}(b),\mathcal{P}_{ijk}(c)\big)$. Similar calculations to the above give
\begin{align*}
\varphi_\Ga(X_1X_2X_3)&=u_a(u_a+1)\cdot (u_b-1)\cdot u_c \cdot x_ay_b,
  & \mathcal{P}_{123}(a,b,c)&=(1,-1,1), \\ 
\varphi_\Ga(X_1X_3X_2)&=u_a(u_a+1)\cdot (u_b+1)\cdot u_c\cdot x_ay_b,
  & \mathcal{P}_{132}(a,b,c)&=(1,1,1), \\ 
\varphi_\Ga(X_2X_1X_3)&=(u_a-2)(u_a-3)\cdot (u_b-1)\cdot u_c\cdot x_ay_b,
  & \mathcal{P}_{213}(a,b,c)&=(-1,-1,1), \\ 
\varphi_\Ga(X_2X_3X_1)&=(u_a-2)(u_a-3)\cdot (u_b-1)\cdot (u_c-1)\cdot x_ay_b,
  & \mathcal{P}_{231}(a,b,c)&=(-1,-1,-1), \\ 
\varphi_\Ga(X_3X_1X_2)&=u_a(u_a+1)\cdot (u_b+1)\cdot (u_c-1)\cdot x_ay_b,
  & \mathcal{P}_{312}(a,b,c)&=(1,1,-1), \\ 
\varphi_\Ga(X_3X_2X_1)&=(u_a-2)(u_a-3)\cdot (u_b+1)\cdot (u_c-1) \cdot x_ay_b,
  & \mathcal{P}_{321}(a,b,c)&=(-1,1,-1). \\ 
\end{align*}
These identities agree with Lemma \ref{lem:varphi-calculation}.
For example, consider the total order $2\prec 1\prec 3$ on $V$.
The right hand side of \eqref{eq:varphi-calculation} is 
\[P_{\ga_{a2},\ga_{a1}}(u_a)\cdot P_{\ga_{b2},\ga_{b3}}(u_b)\cdot 
P_{\ga_{c1},\ga_{c3}}(u_c)\cdot z^{\ga(g)} \] 
Identifying $\Z V$ with $\Z^3$ we have $g=(1,1,1)$ and thus $\ga(g)=a-b\in\Z E$ so $z^{\ga(g)}=x_ay_b$. Substituting the matrix elements of $\ga$ we obtain
\[P_{3,-2}(u_a)\cdot P_{1,-2}(u_b)\cdot P_{-1,1}(u_c)\cdot x_ay_b=
(u_a-2)(u_a-3)\cdot (u_b-1)\cdot u_c\cdot x_ay_b.\] 
This also agrees with \eqref{eq:varphi-proof-identity}.

Let $f:E\to \{-1,1\}$ be given by $(f(a),f(b),f(c))=(1,-1,-1)$. One verifies that for any total order $\prec$ on $V$, $f\neq\mathcal{P}_\prec$. For any order $\prec$ let 
\[e_\prec=
\begin{cases}
a& \mathcal{P}_\prec(a)=-1 \\ 
b& \mathcal{P}_\prec(a)=1, \mathcal{P}_\prec(b)=1, \\ 
c& \text{otherwise}.
\end{cases}
\] 
Then $\mathcal{P}_\prec(e_\prec)=-f(e_{\prec})$ for any total order $\prec$ on $V$. The ideal $J\subseteq R_E$ from \eqref{eq:ideal-J} is given by 
\[J=R_E(u_a-2)(u_a-3) + R_E(u_b+1) + R_Eu_c\] 
which is a proper ideal of $R_E$, and
\[\varphi_\Ga(X_iX_jX_k)\subseteq J \cdot x_ay_b\]
for $\{i,j,k\}=\{1,2,3\}$. This shows that $\varphi_\Ga$ is not locally surjective.
\end{example}

\subsection{Universality of the pair \texorpdfstring{$(\mathcal{A}(\Gamma), \varphi_\Gamma)$}{(A(Gamma),phiGamma)}}
In this section we show that the algebras $\mathcal{A}(\Gamma)$ are universal among all TGW algebras (with polynomial base ring) that can be represented by differential operators.

\begin{thm}\label{thm:main}
Let $\mathcal{A}=\TGWA{\mu}{R}{\si}{t}$ be any TGW algebra with index set denoted $V$, such that $R$ is a polynomial algebra $R=R_E=\K[u_e\mid e\in E]$ (for some index set $E$),
and $\mu$ is symmetric. Assume that
\begin{equation}
\varphi:\mathcal{A}\to A_E(\K)
\end{equation}
is a map of $R_E$-rings with involution, where $A_E(\K)$ is the Weyl algebra over $\K$ with index set $E$.
Then $\mathcal{A}$ is consistent and
there exists a multiquiver $\Gamma=(V,E,s,t)$ with vertex set $V$ and edge set $E$ and a map
\[\xi:\mathcal{A}\to\mathcal{A}(\Gamma)\]
of $\Z V$-graded $R_E$-rings with involution
such that the following diagram commutes: 
\begin{equation}\label{eq:univ-diagram}
\begin{aligned}
\begin{tikzcd}[ampersand replacement=\&,
               column sep=small]
\mathcal{A}
 \arrow{rr}{\varphi}
 \arrow{d}[swap]{\xi} 
 \& \& A_E(\K) \\ 
\mathcal{A}(\Gamma)
 \arrow{rru}[swap]{\varphi_{\Gamma}}
 \& \& 
\end{tikzcd}
\end{aligned}
\end{equation}
Moreover, if $\varphi(X_i)\neq 0$ for each $i\in V$, then $\Gamma$ is uniquely determined and $\mu_{ij}=1$ for all $i,j\in V$. 

\end{thm}

\begin{proof}
(a) Since $\varphi$ is a map of $R_E$-rings, $\mathcal{A}$ is consistent.
We will define the graph $\Gamma$ by specifying its incidence matrix 
$\gamma=(\gamma_{ev})_{e\in E, v\in V}$.
Let $v\in V$. If $\varphi(X_v)=0$ (or equivalently $\varphi(Y_v)=0$, since $\varphi$ is a map of rings with involution) then we define 
$\gamma_{ev}=0$ for all $e\in E$.
Assume $\varphi(X_v)\neq 0$.
For each $i\in V$ we have
\begin{equation}\label{eq:pf_alpha1}
\varphi(X_v)=\sum_{g\in \Z E} z^{(g)}s_{v,g},\quad\forall i\in\iv{1}{m},
\end{equation}
for some $s_{v,g}\in R_E$, only finitely many nonzero and $s_{v,g}\neq 0$ for at least one $g$.
Let $e\in E$ and act by $\ad u_e$ on both sides of \eqref{eq:pf_alpha1} to get
\begin{equation}\label{eq:pf_alpha2}
(u_e-\si_v(u_e))\varphi(X_v)=\sum_{g\in\Z E} g_ez^{(g)}s_{v,g},\quad\forall v\in V,\forall e\in E.
\end{equation}
Multiplying both sides of \eqref{eq:pf_alpha1} by $u_e-\si_v(u_e)$ from the left and subtracting the resulting equation from \eqref{eq:pf_alpha2}, and using the $\Z E$-gradation on $A_E(\K)$, we get
\begin{equation}
\big(g_e-u_e+\si_v(u_e)\big)z^{(g)}s_{v,g}=0,\quad\forall g\in\Z E,\forall v\in V,\forall e\in E.
\end{equation}
So if $s_{v,g}\neq 0$ for some $g\in\Z E$, then, since $A_E(\K)$ is a domain, 
$g_e-u_e+\si_v(u_e)=0 \forall e\in E$.
This shows that there exists exactly one $\gamma_v\in\Z E$ such that $s_{v,\gamma_v}\neq 0$. Namely, $\ga_v=\sum_{e\in E} \ga_{ev}e$ where
\begin{equation}\label{eq:pf_alpha3}
\si_v(u_e)=u_e-\ga_{ev},\quad \forall v\in V,\forall e\in E.
\end{equation} 
This defines a matrix $\ga=(\ga_{ev})_{e\in E, v\in V}$ with integer entries. From the above it follows that the column $(\ga_{ev})_{e\in E}$ is uniquely determined if $\varphi(X_v)\neq 0$. Since $\ga_v\in \Z E$, condition \eqref{eq:condition_C} holds.

Next we show that $\ga$ satisfies condition \eqref{eq:condition_M}.
For $i\in V$, put $s_i=s_{i,\ga_i}$ if $\varphi(X_i)\neq 0$ and $s_i=0$ otherwise. We have shown that
\begin{equation}
\varphi(X_i)= z^{(\ga_i)} s_i,\quad \varphi(Y_i)=\varphi(X_i)^\ast=s_iz^{(-\ga_i)}.
\end{equation}
Let $i,j\in V$, $i\neq j$, and assume $s_i,s_j\neq 0$. Applying $\varphi$ to the relation $X_iY_j=\mu_{ij}Y_jX_i$ we obtain
\[z^{(\ga_i)}s_is_jz^{(-\ga_j)}=\mu_{ij}s_jz^{(-\ga_j)}z^{(\ga_i)}s_i.\] 
Using that $\varphi$ is a map of $R_E$-rings we get
\begin{equation}\label{eq:univ-eq1}
\si_i(s_is_j)z^{(\ga_i)}z^{(-\ga_j)}=\mu_{ij}s_j\si_j^{-1}\si_i(s_i)z^{(-\ga_j)}z^{(\ga_i)}.
\end{equation}
We have
\begin{equation}
z^{(\ga_i)}z^{(-\ga_j)}=P_{ij} z^{(\ga_i-\ga_j)}
\end{equation}
and
\begin{equation}
z^{(-\ga_j)}z^{(\ga_i)}=Q_{ij} z^{(\ga_i-\ga_j)}
\end{equation}
where $P_{ij}=\prod_{e\in E} P_{e;ij}$ and $Q_{ij}=\prod_{e\in E} Q_{e;ij}$ where $P_{e;ij}, Q_{e;ij}\in \K[u_e]$. 
Substituting into \eqref{eq:univ-eq1}, canceling $z^{(\ga_i-\ga_j)}$ we obtain:
\begin{equation}\label{eq:univ-eq2}
s_i's_j' P_{ij} = \mu_{ij} \si_i^{-1}(s_j')\si_j^{-1}(s_i') Q_{ij}
\end{equation}
where $s_i'=\si_i(s_i)$ and $s_j'=\si_i(s_j)$.

Suppose for the sake of contradiction that $\gamma$ does not satisfy condition
\eqref{eq:condition_M}. Then there exist $e\in E$, $i,j\in V$, $i\neq j$, such that $\ga_{ei}\ga_{ej}>0$. We consider the case when $\ga_{ei}$ and $\ga_{ej}$ are both negative, the case when both are positive being treated analogously. 
Since $z_e^{(\ga_{ei})}z_e^{(-\ga_{ej})}=P_{e;ij}z_e^{(\ga_{ei}-\ga_{ej})}$
and 
$z_e^{(-\ga_{ej})}z_e^{(\ga_{ei})}=Q_{e;ij}z_e^{(\ga_{ei}-\ga_{ej})}$,
Lemma \ref{lem:Pmn} then implies that
\begin{equation}\label{eq:univ-key1}
\text{All roots of $P_{e;ij}$ are non-positive and all roots of $Q_{e;ij}$ are positive.}
\end{equation}
Let $S$ be the set of integers $k$ satisfying the following two conditions:
\begin{gather}
\label{eq:univ-div1}
(u_e-k)\big| s_i's_j'P_{ij},\\ 
\label{eq:univ-div2}
(u_e-k)\not\big| P_{ij}.
\end{gather}
Clearly $S$ is finite, due to \eqref{eq:univ-div1}. The set $S$ is nonempty, since it contains all roots of $Q_{e;ij}$, by \eqref{eq:univ-eq2} and \eqref{eq:univ-key1}. 
Let $k_0$ denote the largest element of $S$. In particular $k_0>0$.
From \eqref{eq:univ-div1}-\eqref{eq:univ-div2}, $u_e-k_0$ divides $s_i'$ or $s_j'$. Applying $\si_j^{-1}$ respectively $\si_i^{-1}$, using \eqref{eq:pf_alpha3}, we obtain that 
\[(u_e+\ga_{ej}-k_0)\big| \si_j^{-1}(s_i)
\quad\text{or}\quad 
(u_e+\ga_{ei}-k_0)\big| \si_i^{-1}(s_j).\] 
By \eqref{eq:univ-eq2} we get
\[ 
(u_e+\ga_{ea}-k_0)\big| s_is_jP'_{ij}
\]
for some $a\in \{i,j\}$.
By \eqref{eq:univ-key1}, $-\ga_{ea}+k_0$ is not a root of $P_{e;ij}$. Therefore $-\ga_{ea}+k_0\in S$ which contradicts the maximality of $k_0$.
This contradiction shows that $\gamma$ satisfies condition \eqref{eq:condition_M}.
In particular, using this together with \eqref{eq:univ-eq1} we have
\begin{equation}
\label{eq:univ-eq3}
[z^{(\ga_i)},z^{(-\ga_j)}]=0\quad\forall i,j\in V, i\neq j.
\end{equation}
By \eqref{eq:univ-eq3} and \eqref{eq:univ-eq1} we get
\begin{equation}\label{eq:univ-eq4}
\si_i(s_is_j)=\mu_{ij}s_j\si_j^{-1}\si_i(s_i)
\qquad\forall i,j\in V, i\neq j.
\end{equation}
Giving $E$ any order, the polynomial ring $R_E$ becomes filtered and the automorphisms $\si_i$ act as the identity on the associated graded algebra. Therefore, taking leading terms on both sides of \eqref{eq:univ-eq4} we get that
\begin{equation}\label{eq:univ-eq5}
\mu_{ij}=1\quad\text{if $s_is_j\neq 0$.}
\end{equation}
Next we construct the map $\xi:\mathcal{A}\to\mathcal{A}(\Ga)$. Let $\mathcal{X}=\{X_v\mid v\in V\}\cup \{Y_v\mid v\in V\}$ and define $\xi:\mathcal{X}\to\mathcal{A}(\Ga)$ by
\begin{equation}\label{eq:xi-def}
\xi(X_v)=X_v^\Ga s_v,\quad \xi(Y_v)=s_v Y_v^\Ga,\qquad \forall v\in V,
\end{equation}
where $X_v^\Ga$ and $Y_v^\Ga$ denote the generators of $\mathcal{A}(\Ga)$.
Extend $\xi$ uniquely to a map $\xi:F_{R_E}(\mathcal{X})\to\mathcal{A}(\Ga)$ where $F_{R_E}(\mathcal{X})$ is the free $R_E$-ring on the set $\mathcal{X}$. We must verify that the elements \eqref{eq:tgwarels} are in the kernel of $\xi$. By \eqref{eq:pf_alpha3} and \eqref{eq:sigma-definition} one checks that
\[
\xi(X_v u_e-\si_v(u_e)X_v)=
X_v^\ga s_v u_e- (u_e-\ga_{ev}) X_v^\Ga s_v=0
,\quad\forall v\in V,\;\forall e\in E,\\ 
\] 
and thus
\begin{equation}
\label{eq:univ-check1}
\xi(X_v r-\si_v(r)X_v)=0,
\qquad\forall v\in V,\;\forall r\in R_E.
\end{equation}
Similarly one verifies that
\begin{equation}
\label{eq:univ-check1b}
\xi(Y_v r-\si_v^{-1}(r)Y_v)=0,
\qquad\forall v\in V,\;\forall r\in R_E.
\end{equation}
Let $i,j\in V$, $i\neq j$. Denoting the generators of $\mathcal{A}(\Ga)$ by $X_i^\Ga, Y_i^\ga$ we have
\begin{equation}
\label{eq:univ-check2}
[\xi(X_i),\xi(Y_j)]=[X_i^\Ga s_i,s_jY_j^\Ga]=
(\si_i(s_is_j)-s_j\si_j^{-1}\si_i(s_i))X_i^\Ga Y_j^\Ga
\end{equation}
which is zero by \eqref{eq:univ-eq4} and \eqref{eq:univ-eq5}.
For $i\in V$, put $t_i^\Ga=Y_i^\Ga X_i^\Ga\in R_E$. Note $t_i$ and $t_i^\Ga$ are related as follows:
\begin{align*} 
t_i&=\varphi(t_i)=\varphi(Y_iX_i)=\varphi(Y_i)\varphi(X_i)=s_iz^{(-\ga_i)}z^{(\ga_i)}s_i=\\
&=s_i\varphi_\Ga(Y_i^\Ga)\varphi_\Ga(X_i^\Ga)s_i=s_i^2\varphi_\Ga(t_i^\Ga)=s_i^2 t_i^\Ga.
\end{align*}
Therefore 
\begin{equation}
\label{eq:univ-check3}
\xi(Y_iX_i-t_i)=\xi(Y_i)\xi(X_i)-\xi(t_i)=s_iY_i^\Ga X_i^\Ga s_i - t_i = s_i^2 t_i^\Ga-t_i=0,
\qquad\forall i\in V.
\end{equation}
Similarly one checks that
\begin{equation}
\label{eq:univ-check4}
\xi(X_iY_i-\si(t_i))=0,\qquad\forall i\in V.
\end{equation}
This shows that $\xi$ descends to a map of $R_E$-rings with involution $\xi:\TGWC{\mu}{R_E}{\si}{t}\to \mathcal{A}(\Gamma)$. As in the last part of the proof of Theorem \ref{thm:map_to_Weyl_algebra}, we may use the gradation form on $\TGWC{\mu}{R_E}{\si}{t}$ and that $\mathcal{A}(\Gamma)$ is domain (by Corollary \ref{cor:domain}) to conclude that $\xi$ descends to a map of $R_E$-rings with involution $\xi:\mathcal{A}\to\mathcal{A}(\Gamma)$. From the definition \eqref{eq:xi-def} of $\xi$ it follows that $\xi$ is a map of $\Z V$-graded algebras and that the diagram \eqref{eq:univ-diagram} is commutative.
\end{proof}

\section{Relation to enveloping algebras}\label{sec:relation-to-enveloping-algebras}

In this section we assume that $\K$ is an algebraically close field of characteristic zero.

\subsection{Primitivity of \texorpdfstring{$\mathcal{A}(\Ga)$}{A(Gamma)}}
\label{sec:primitivity}
We give a sufficient condition for $\mathcal{A}(\Ga)$ to be primitive. 
\begin{thm}\label{thm:primitivity}
Let $\Ga$ be a multiquiver such that no connected component of $\Ga$ is in equilibrium.
Then $\mathcal{A}(\Ga)$ is a primitive ring.
\end{thm}
\begin{proof}
By Theorem \ref{thm:faithfulness}, the incidence matrix $\ga$ is injective.
Let $E$ and $V$ be the edge and vertex set of $\Ga$ respectively. Recall the definition $\mathcal{A}(R_E,\si^\Ga,t^\Ga)$ of $\mathcal{A}(\Ga)$ as a TGW algebra.
Consider the maximal ideal $\mathfrak{m}_{1/2}=(u_e-1/2\mid e\in E)$
of $R_E$. For $d\in\Z V$ we have $\si_d(\mathfrak{m}_{1/2})=(u_e-\ga(d)_e-1/2\mid e\in E)$. Thus $\si_d(\mathfrak{m}_{1/2})=\mathfrak{m}_{1/2} \Leftrightarrow \ga(d)=0 \Leftrightarrow d=0$ since $\ga$ is injective. Therefore, by \cite[Prop.~7.2]{MazPonTur2003}, the induced module $\mathcal{A}(\Ga)\otimes_{R_E} R_E/\mathfrak{m}_{1/2}$ has a unique simple quotient $M$ with the weight space $M_{\mathfrak{m}_{1/2}}\neq 0$. Since $t_i$ is invertible modulo $\mathfrak{m}_{1/2}$, $X_i$ and $Y_i$ act injectivly on $M$ for all $i\in V$, and thus the support of $M$ is
$\Supp(M)=\{\si_d(\mathfrak{m}_{1/2})\mid d\in \Z V\}\simeq \Z V$.
Therefore $\Ann_{R_E}M=\bigcap_{\mathfrak{m}\in\Supp(M)} \mathfrak{m}=\{0\}$.
Futhermore, $M$ is a graded $\mathcal{A}(\Ga)$-module via $M=\bigoplus_{d\in \Z V} M_{\si_d(\mathfrak{m}_{1/2})}$. Hence $I:=\Ann_{\mathcal{A}(\Ga)}M$ is a graded ideal of $\mathcal{A}(\Ga)$. If $I\neq\{0\}$, then $I\cap R_E\neq\{0\}$ since $\mathcal{A}(\Ga)$ is a TGW algebra. But $I\cap R_E=\Ann_{R_E}M=\{0\}$ as we just saw. This proves that $I=0$ and thus $M$ is a faithful simple $\mathcal{A}(\Ga)$-module. Hence $\mathcal{A}(\Ga)$ is primitive.
\end{proof}

\subsection{The generalized Cartan matrix of \texorpdfstring{$\mathcal{A}(\Gamma)$}{A(Gamma)}}
\label{sec:theGCM}

In \cite{Hartwig2010} it was found that one can associate generalized Cartan matrices to TGW data. These matrices do not completely determine the TGW algebra (as they do for Kac-Moody algebras) but they can be used to describe certain relations in the algebra.

Let $(R,\si,t)$ be a TGW datum with index set $I$. For $i,j\in I$, put
\begin{align}
V_{ij}&=\text{span}_\K \{\si_i^k(t_j)\mid k\in\Z\}, \\ 
p_{ij}(x) &=\text{minimal polynomial for $\si_i$ acting on $V_{ij}$ (if $\dim_\K V_{ij}<\infty$)},\\ 
a_{ij} &=
\begin{cases}
2,& i=j\\ 
1-\dim_\K V_{ij}, & i\neq j
\end{cases} \\ 
C&=(a_{ij})_{i,j\in I}
\end{align}
If $(R,\si,t)$ is \emph{regular} (i.e. $t_i$ is not a zero-divisor in $R$ for each $i\in I$) and \emph{$\K$-finitistic} (i.e. $V_{ij}$ is finite-dimensional for all $i,j\in I$) then $C$ is a generalized Cartan matrix, by \cite[Thm.~4.4]{Hartwig2010}. Moreover, the (not necessarily commuting) generators $X_i, X_j$ (respectively $Y_i,Y_j$) satisfy certain homogenous Serre-type relations expressed using the coefficients of the polynomials $p_{ij}(x)$.

In the following lemma we compute the polynomials $p_{ij}(x)$ and integers $a_{ij}$ explicitly for the TGW algebras $\mathcal{A}(\Ga)$ associated to a multiquiver $\Ga$.

\begin{thm}\label{thm:GCM}
Let $\Ga=(V,E,s,t)$ be a multiquiver, $\ga=(\ga_{ei})_{e\in E, i\in V}$ be its incidence matrix, and $(R_E,\si^\Ga,t^\Ga)$ the corresponding TGW datum. Then $(R_E,\si^\Ga,t^\Ga)$ is regular and $\K$-finitistic. Moreover, for 
$i,j\in V$ with $i\neq j$ we have
\begin{align} 
\label{eq:pij_from_gamma}
p_{ij}(x) &= (x-1)^{1-a_{ij}}, \\ 
\label{eq:GCM_from_gamma}
a_{ij} &= -\sum_{\substack{e\in E \\  \ga_{ei}\neq 0}} |\ga_{ej}|.
\end{align}
\end{thm}
\begin{proof} Fix $i,j\in V$, $i\neq j$. Put
\[ N_{ij}=\sum_{\substack{e\in E\\ \ga_{ei}\neq 0}} |\ga_{ej}|. \] 
Consider the difference operator
\[ D_i=\si_i-\Id \] 
and let $m$ be a non-negative integer. By definition of $t_j$ we have
\begin{equation}\label{eq:Dynkin-lem-pf0a}
D_i^{m}(t_j) = D_i^m\Big(\prod_{e\in E} u_{ej}\Big).
\end{equation}
Since $u_{ej}$ is a monic polynomial in $u_e$, $D_i(u_e a)=u_e D_i(a)$ if $\ga_{ei}=0$, and $u_{ej}=1$ if $\ga_{ej}=0$, we have
\begin{equation}\label{eq:Dynkin-lem-pf0b}
D_i^m\Big(\prod_{e\in E} u_{ej}\Big) =
\Big(\prod_{e\in E\setminus E_{ij}} u_{ej}\Big)\cdot D_i^m\Big(\prod_{e\in E_{ij}} u_{ej}\Big)
\end{equation}
where $E_{ij}=\big\{e\in E\mid \ga_{ei}\neq 0\neq \ga_{ej}\big\}$ is the set of edges between the vertices $i$ and $j$. Using the twisted Leibniz rule
\begin{equation}\label{eq:Di-Leibniz}
D_i(ab)=D_i(a)\si_i(b)+aD_i(b),\quad\forall a,b\in R_E,
\end{equation}
the second factor in \eqref{eq:Dynkin-lem-pf0b} can be computed as the sum over all possible ways of distributing the $m$ difference operators $D_i$ among the factors $u_{ej}$, $e\in E_{ij}$ (and in addition applying $\si_i$ to some factors).
If $m=1+N_{ij}$ then, by the pigeon hole principle, any such term will contain at least one factor of the form $D_i^n(u_{ej})$ where $n\ge 1+|\ga_{ej}|$. Since $u_{ej}$ is a polynomial in $u_e$ of degree $|\ga_{ej}|$, \eqref{eq:Di-Leibniz} and $D_i(u_e)=-\ga_{ei}$ imply that $D_i^n(u_{ej})=0$. This shows that
\begin{equation}
D_i^{1+N_{ij}}(t_j)=0.
\end{equation}
Since $D_i$ and $\si_i$ commute, we get $D_i^{1+N_{ij}}\big|_{V_{ij}}=0$. Hence the minimal polynomial $p_{ij}(x)$ for $\si_i\big|_{V_{ij}}$ divides $(x-1)^{1+N_{ij}}$.

It remains to be shown that $D_i^{N_{ij}}(t_j)\neq 0$.
By the above argument, $D_i^m(p)=0$ if $p$ is a polynomial in $u_e$ of degree less than $m$. Hence we may replace $u_{ej}$ by its leading term when applying $D_i^{|\ga_{ej}|}$:
\begin{equation}\label{eq:Dynkin-lem-pf0c}
D_i^{|\ga_{ej}|}(u_{ej})=D_i^{|\ga_{ej}|}(u_e^{|\ga_{ej}|}).
\end{equation}
By iterating \eqref{eq:Di-Leibniz} we get
\[ 
D_i(u_e^m)=m u_e^{m-1}\cdot (-\ga_{ei})+\text{lower terms},\quad\forall e\in E.
\]
Thus
\begin{equation} \label{eq:Dynkin-lem-pf0d}
D_i^m(u_e^m)=m!\cdot(-\ga_{ei})^m.
\end{equation}
Combining \eqref{eq:Dynkin-lem-pf0a}-\eqref{eq:Dynkin-lem-pf0b} and \eqref{eq:Dynkin-lem-pf0c}-\eqref{eq:Dynkin-lem-pf0d} with the fact that there is at most one non-zero way to distribute $N_{ij}$ difference operators $D_i$ among the factors $u_{ej}$, $e\in E_{ij}$, we get
\begin{equation}
D_i^{N_{ij}}(t_j)=
  \prod_{e\in E\setminus E_{ij}} u_{ej} \cdot 
  \prod_{e\in E_{ij}} D_i^{|\ga_{ej}|}(u_{ej})  \\ 
= \prod_{e\in E\setminus E_{ij}} u_{ej} \cdot
  \prod_{e\in E_{ij}} |\ga_{ej}|!\cdot (-\ga_{ei})^{|\ga_{ej}|} \neq 0.
\end{equation}
This proves that $(\si_i-\Id)^m(t_j)\neq 0$ for any $m\le N_{ij}$. Hence $p_{ij}(x)=(x-1)^{1+N_{ij}}$ which finishes the proof.
\end{proof}

Theorem \ref{thm:GCM} leads us to the following procedure.
Let $\Ga$ be a multiquiver. One can directly associate to $\Ga$ a Dynkin diagram $D(\Ga)$ by applying the following steps:
\begin{enumerate}[{\rm (i)}]
\item Remove all leaves;
\item Forget the direction of each edge;
\item For each pair of adjacent vertices $i$ and $j$, replace the collection of edges between them by a single edge, adding multiplicities in the process:
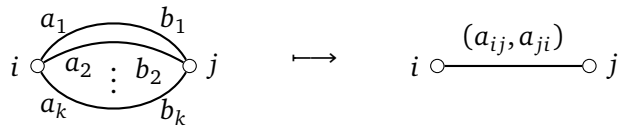
\[
\begin{tikzpicture}[baseline=(X.base)]
\node (X) at (0,0) {};
\node[vertex] (1) at (0,0) [label=left:$i$] {};
\node[vertex] (2) at (2,0) [label=right:$j$] {};
\draw[thick] (1) to[out=60,in=120] node[above,very near start] {$a_1$} node[above,very near end] {$b_1$} (2);
\draw[thick] (1) to[out=30,in=150]
   node[below,near start] {$a_2$} node[below] {$\vdots$} node[below, near end] {$b_2$} (2);
\draw[thick] (1) to[out=-60,in=-120] node[below,very near start] {$a_k$} node[below,very near end] {$b_k$} (2);
\end{tikzpicture}
 \qquad \longmapsto \qquad 
\begin{tikzpicture}[baseline=(X.base)]
\node (X) at (0,0) {};
\node[vertex] (1) at (0,0)[label=left:$i$] {};
\node[vertex] (2) at (2,0) [label=right:$j$] {};
\draw[thick] (1) -- node[auto] {$(a_{ij},a_{ji})$} (2);
\end{tikzpicture}
\] 
where $a_{ij}=\sum_{e\in E_{ij}} |\ga_{ei}|=\sum_{m=1}^k a_m$ and
$a_{ji}=\sum_{e\in E_{ij}} |\ga_{ej}|=\sum_{m=1}^k b_m$, where $E_{ij}\subseteq E$ is the set of edges in $\Ga$ between $i$ and $j$.
\end{enumerate}
As is well-known (see for example \cite{Kac}), Dynkin diagrams are in bijection with generalized Cartan matrices $C=(a_{ij})_{i,j\in I}$ consisting of integers satisfying $a_{ij}\le 0$ for $i\neq j$, $a_{ii}=2$ and $a_{ij}=0$ iff $a_{ji}=0$.

\begin{thm}\label{thm:commutative-diagram}
The following diagram of assignments is commutative.
\[ 
\begin{tikzcd}[ampersand replacement=\&, column sep=small]
\text{multiquiver} \& \Ga\arrow[mapsto]{rr} \arrow[mapsto]{d}\&\& D(\Ga) \arrow[mapsto]{d}\& \text{Dynkin diagram}\\ 
\text{TGW datum}\& (R_E,\si^\Ga,t^\Ga) \arrow[mapsto]{rr}\arrow[mapsto]{d}\&\& C \arrow[mapsto]{d}\& \text{GCM}\\ 
\text{TGW algebra} \& \mathcal{A}(\Ga) \&\& \mathfrak{g}(C)\& \text{Kac-Moody algebra}
\end{tikzcd}
\]
\end{thm}
\begin{proof}
The statement follows directly from Theorem \ref{thm:GCM} and the definition of the Dynkin diagram $D(\Ga)$ associated to $\Ga$.
\end{proof}

\begin{rem}
In \cite{Hartwig2010}, it was shown that 
the image of the map $(R,\si,t)\mapsto C$ from TGW data to GCMs contains all symmetric GCMs. It was further asked if this image contains any non-symmetric GCMs. The above shows that the map is in fact surjective.
Indeed, it is easy to see that the map $D:\Ga\mapsto D(\Ga)$ from multiquivers to Dynkin diagrams is surjective (but not injective). Therefore, by Theorem \ref{thm:commutative-diagram}, $(R,\si,t)\mapsto C$ is surjective.
\end{rem}

The following theorem gives a partial connection between the algebras $\mathcal{A}(\Ga)$ and the Kac-Moody algebras $\mathfrak{g}(C)$.

\begin{thm}\label{thm:serre-relations}
Let $\Ga$ be a multiquiver, $\mathcal{A}(\Ga)$ be the corresponding TGW algebra, and $\mathfrak{g}$ be the Kac-Moody algebra associated to $D(\Ga)$. Let $\mathfrak{n}^+$ (respectively $\mathfrak{n}^-$) be the positive (respectively negative) subalgebra of $\mathfrak{g}$ with generators $e_i$ (respectively $f_i$) $i\in V$.
Then there exist graded $\K$-algebra homomorphisms
\begin{equation}\label{eq:psi_pm}
\begin{aligned}
\psi^+:U(\mathfrak{n}^+) &\to \mathcal{A}(\Ga),\quad e_i\mapsto X_i,\\
\psi^-:U(\mathfrak{n}^-) &\to \mathcal{A}(\Ga),\quad f_i\mapsto Y_i.
\end{aligned}
\end{equation}
Moreover, for any $i,j\in V$, $i\neq j$, and $k\in\iv{0}{-a_{ij}}$, the restrictions
\[\psi^+|_{U(\mathfrak{n}^+)_{k\al_i+\al_j}}\qquad\text{and}\qquad  
  \psi^-|_{U(\mathfrak{n}^-)_{-k\al_i-\al_j}}\] 
are injective, where $\{\al_i\}_{i\in V}$ is the set of simple roots of $\mathfrak{g}$.
\end{thm}
\begin{proof}
The statement follows from Theorem \ref{thm:GCM} and \cite[Thm~4.4(b)]{Hartwig2010}.
\end{proof}

\subsection{General and special linear Lie algebras}
\label{sec:special-linear-lie-algebras}

Let $\widetilde{A_n}$ be the following multiquiver:
\[ 
\widetilde{A_n}:\;
\begin{tikzpicture}[baseline=(etc.base)]
\node (start) at (0,0) {};
\node[vertex] (1) at (1,0) [label=below:$1$] {};
\node[vertex] (2) at (2,0) [label=below:$2$] {};
\node[vertex] (3) at (3,0) [label=below:$3$] {};
\node[vertex] (N) at (4,0) [label=below:$n$] {};
\node (end) at (5,0) {};
\node (etc) at (3.5,0) {$\cdots$};
\draw[leaf] (start) -- (1);
\draw[edge] (1) -- (2);
\draw[edge] (2) -- (3);
\draw[leaf] (N) -- (end);
\end{tikzpicture}
\qquad\qquad
\ga=\begin{bmatrix}
1  &    &        && \\
-1 & 1  &        &&\\
   & -1 & \ddots       &  &\\
   &    & \ddots & \ddots   & \\
   &    &        & -1 & 1 \\
   &    &        &    & -1 
\end{bmatrix}
\] 
We identify $V=\iv{1}{n}$ with the set of simple roots of $\mathfrak{sl}_{n+1}$ and regard the edge set as $E=\iv{1}{n+1}$.
Let $\mathbb{E}=\sum_{i\in E} x_iy_i= \sum_{i\in E} (u_i-1)$ be the Euler operator in the Weyl algebra $A_E(\K)$. Let $A_E(\K)^\mathbb{E}=\{a\in A_E(\K)\mid [\mathbb{E},a]=0\}$ be the invariant subalgebra with respect to the adjoint action of $\mathbb{E}$.
The following lemma is well-known.
\begin{lem}\label{lem:sln+1}
There exists a homomorphism
\[ \pi: U(\mathfrak{gl}_{n+1})\to A_E(\K)\] 
given by
\[ 
\pi(e_i)=x_iy_{i+1},\quad 
\pi(f_i)=x_{i+1}y_i,\quad 
\pi(E_{jj})=u_j,
\] 
for $i\in \iv{1}{n}$ and $j\in\iv{1}{n+1}$, where $e_i,f_i$ are Chevalley generators and $E_{ii}$ are the diagonal matrix units of $\mathfrak{gl}_{n+1}$. Moreover, the image of $\pi$ coincides with $A_E(\K)^\mathbb{E}$.
\end{lem}

The next result says that $\pi$ factors through the canonical representation $\varphi_{\Ga}$ of the TGW algebra $\mathcal{A}(\Ga)$ where $\Ga=\widetilde{A_n}$. For $n=3$ we recover the case considered by A. Sergeev \cite{Sergeev2001}.

We will regard $U(\mathfrak{gl}_{n+1})$ as an $R_E$-ring via
\[R_E\to U(\mathfrak{gl}_{n+1}),\quad u_i\mapsto E_{ii}. \] 
\begin{prp}
There exists a homomorphism of $\Z V$-graded $R_E$-rings
\[\psi:U(\mathfrak{gl}_{n+1})\to \mathcal{A}(\widetilde{A_n})\] 
given by
\[\psi(e_i)=X_i,\quad \psi(f_i)=Y_i,\quad \psi(E_{jj})=u_j,\]
for $i\in\iv{1}{n}$ and $j\in\iv{1}{n+1}$.
Moreover, we have the following commutative diagram:
\begin{equation}
\begin{tikzcd}[ampersand replacement=\&,
               column sep=small]
U(\mathfrak{gl}_{n+1})
 \arrow[twoheadrightarrow]{rrd}{\pi}
 \arrow[twoheadrightarrow]{d}[swap]{\psi} 
 \& \&  \\ 
\mathcal{A}(\widetilde{A_n})
 \arrow{rr}{\simeq}[swap]{\varphi_{\widetilde{A_n}}}
 \& \& A_E(\K)^{\mathbb{E}} 
\end{tikzcd}
\end{equation}
\end{prp}
\begin{proof}
The multiquiver $\widetilde{A_n}$ is connected and not in equilibrium since it has leaves. Thus by Theorem \ref{thm:faithfulness},  $\varphi_{\widetilde{A_n}}$ is faithful. Since 
$\widetilde{A_n}$ has no cycles, by Theorem \ref{thm:local-surjectivity}, the image of $\varphi_{\widetilde{A_n}}$ equals $\bigoplus_{d\in \Z E} (A_E(\K))_{\ga(d)}$, which coincides with $A_E(\K)^\mathbb{E}$ since $\ga(\Z^n)=\{g=\sum_e g_ee\in \Z E\mid \sum_e g_e=0\}$.
It only remains to note that $\psi$ coincides with $\varphi_{\widetilde{A_n}}^{-1}\circ \pi$ on the generators $e_i, f_i, E_{jj}$.
\end{proof}

For $\la\in\K$, let $\mathcal{A}^\la(\widetilde{A_n})$ denote the quotient algebra $\mathcal{A}(\widetilde{A_n})/\langle \mathbb{E}-\lambda\rangle$. Since $\varphi_{\widetilde{A_n}}$ is faithful and locally surjective, each homogeneous component $\mathcal{A}(\widetilde{A_n})_g$ is a cyclic left $R_E$-module generated by $(\varphi_{\widetilde{A_n}})^{-1}(z^{\ga(g)})$. By Proposition \ref{prp:quotients}, $\mathcal{A}^\la(\widetilde{A_n})$ is isomorphic to the TGW algebra $\mathcal{A}(R_E^\la, \bar \sigma, \bar t)$ where $R_E^\la=R_E/\langle \mathbb{E}-\la\rangle$ and $\bar\sigma_i$ is induced from $\sigma_i$, and $\bar t_i$ is the image of $t_i$. Note that $U(\mathfrak{sl}_{n+1})$ becomes an $R_E^\la$-ring via $\bar u_i-\bar u_{i+1}\mapsto h_i$. 

\begin{thm}\label{thm:sln-quotient}
Let $M$ be an infinite-dimensional completely pointed $\mathfrak{sl}_{n+1}$-module and let $J=\Ann_{U(\mathfrak{sl}_{n+1})} M$. Then there exists a $\la\in\K$ such that the primitive quotient $U(\mathfrak{sl}_{n+1})/J$ is isomorphic as $\Z V$-graded $R_E^\la$-rings to $\mathcal{A}^\la(\widetilde{A_n})$.
\end{thm}
\begin{proof}
By \cite{BenBriLem1997}, there exists a simple weight module $W$ over the Weyl algebra $A_E(\K)$ and a scalar $\la\in\K$ such that $M$ is the $\la$-eigenspace in $W$ with respect to the Euler operator $\mathbb{E}$.
Moreover the algebra $B_E^\la$ of twisted differential operators on $\mathbb{P}^n$ given by $B_E^\la=A_E(\K)^\mathbb{E}/\langle\mathbb{E}-\la\rangle$ acts faithfully on $M$.
Indeed, since $B_E^\la$ is a domain, it is sufficient to show
that $\operatorname{Ann}_{R_E}M=0$. Recall from \cite{Mathieu2000} that
there exists the set $S$ of commuting roots such that $S$ spans
the whole weight space and a root vector $X_\beta$ acts injectively
on $M$ for all $\beta\in S$. Let $M_\mu$ be a weight subspace. Then 
$$\operatorname{Ann}_{R_E}M \subset\bigcap_{n_\beta
  >0}\operatorname{Ann}_{R_E}M_{\mu+\sum_{\beta\in S} n_\beta\beta}=0,$$
where the second equality follows from the condition that $S$ spans the whole space.

Let $\al:U(\mathfrak{sl}_{n+1})\to\End_\K(M)$ denote the representation corresponding to the module $M$.
Then we have the following commutative diagram, where we use $\psi$ and $\pi$ to also denote the restrictions to $U(\mathfrak{sl}_{n+1})$, $\xi_\la, \xi_\la'$ are the canonical projections, and $\bar\varphi_{\widetilde{A_n}}$ is induced from $\varphi_{\widetilde{A_n}}$:

\[ 
\begin{tikzcd}[ampersand replacement=\&,
               column sep=small]
U(\mathfrak{sl}_{n+1})
 \arrow{rr}{\al}
 \arrow{rd}{\pi}
 \arrow{d}{\psi} 
 \& \& \End_\K(M) \\ 
\mathcal{A}(\widetilde{A_n})
 \arrow{r}{\simeq}[swap]{\varphi_{\widetilde{A_n}}}
 \arrow[twoheadrightarrow]{d}{\xi_\la}
 \& A_E(\K)^\mathbb{E} \arrow{ur} \arrow[twoheadrightarrow]{dr}{\xi_\la'} \& \\ 
\mathcal{A}^\la(\widetilde{A_n})
 \ar{rr}{\simeq}[swap]{\bar\varphi_{\widetilde{A_n}}}
 \& \& A_E(\K)^\mathbb{E}/\langle\mathbb{E}-\la\rangle \ar[hookrightarrow]{uu}
\end{tikzcd}
\] 
Let $J_\la = \ker (\xi_\la \circ \psi)$. Then it follows from the above diagram that $J=\ker (\xi_\la'\circ\pi) = J_\la$ and thus that $\mathcal{A}^\la(\widetilde{A_n}) \simeq U(\mathfrak{sl}_n)/J_\la$ as required.

\end{proof}

\subsection{Symplectic Lie algebra}
\label{sec:symplectic-lie-algebras}
Let $\widetilde{C_n}=(V,E,s,t)$ be the following multiquiver:
\[
\widetilde{C_n}:\;\; 
\begin{tikzpicture}[baseline=(etc.base)]
\node (start) at (0,0) {};
\node[vertex] (1) at (1,0) [label=below:$1$] {};
\node[vertex] (2) at (2,0) [label=below:$2$] {};
\node[vertex] (3) at (3,0) [label=below:$3$] {};
\node[vertex] (N-2) at (4,0) [label=below:$n-2$] {};
\node[vertex] (N-1) at (5,0) [label=below:$n-1$] {};
\node[vertex] (N) at (6,0) [label=below:$n$] {};
\node (etc) at (3.5,0) {$\cdots$};

\draw[leaf] (start) -- (1);
\draw[edge] (1) -- (2);
\draw[edge] (2) -- (3);
\draw[edge] (N-2) -- (N-1);
\draw[edge] (N-1) -- node[auto,very near start] {$1$} node[auto,very near end] {$2$} (N);
\end{tikzpicture}
\qquad\qquad 
\ga=\begin{bmatrix}
1  &    &        & && \\
-1 & 1  &        & &&\\
   & -1 & \ddots & &&\\
   &    & \ddots &\ddots  &&\\
   &    &        & -1 &  1 & \\
   &    &        &   & -1 & 2
\end{bmatrix}
\] 
We identify $V=\iv{1}{n}$ with the set of simple roots of $\mathfrak{sp}_{2n}$ relative to a fixed choice of Borel subalgebra. In particular we regard $\Z V$ as the root lattice.
We also identify the edge set $E$ with $\iv{1}{n}$. Let $A_E(\K)=A_n(\K)$ be the Weyl algebra with index set $E$, and let $\ep$ be the automorphism of $A_n(\K)$ defined by $\ep(x_i)=-x_i$ and $\ep(y_i)=-y_i$ for all $i\in E$. Let $A_n(\K)^\ep=\{a\in A_n(\K)\mid \ep(a)=a\}$ be the fixed-point subalgebra of $A_n(\K)$ with respect to $\ep$.

\begin{lem}\label{lem:sp2n-lemma}
There exists a $\K$-algebra homomorphism
\[\pi:U(\mathfrak{sp}_{2n})\to A_n(\K)\] 
given by
\[ 
\begin{aligned}
\pi(e_i) &= x_iy_{i+1} \\ 
\pi(e_n) &= \frac{\sqrt{-1}}{2}x_n^2
\end{aligned}
\qquad 
\begin{aligned}
\pi(f_i) &= x_{i+1}y_i \\ 
\pi(f_n) &= \frac{\sqrt{-1}}{2}y_n^2
\end{aligned}
\qquad 
\begin{aligned}
\pi(h_i) &= u_i-u_{i+1} \\ 
\pi(h_n) &= u_n-\frac{1}{2}
\end{aligned}
\] 
for $i\in \iv{1}{n-1}$.
Moreover, the image of $\pi$ coincides with the subalgebra $A_n(\K)^\ep$.
\end{lem}
\begin{proof}
This is well-known, see for example \cite{Dixmier}. The normalization involving $\sqrt{-1}$ is chosen to make $\pi$ respect the natural involutions.
\end{proof}

Consider $U(\mathfrak{sp}_{2n})$ as an $R_E$-ring via
\begin{align*}
R_E &\to U(\mathfrak{sp}_{2n}), \\ 
u_i-u_{i+1} &\mapsto h_i, \\ 
u_n-\frac{1}{2} &\mapsto h_n,
\end{align*}
and equip $U(\mathfrak{sp}_{2n})$ with the $\K$-linear Cartan involution on $U(\mathfrak{sp}_{2n})$ satisfying $e_i^\ast=f_i$, $f_i^\ast=e_i$, $h_i^\ast=h_i$ for all $i$. Then the map $\pi$ from Lemma \ref{lem:sp2n-lemma} is a map of $R_E$-rings with involution.

\begin{prp}
There exists a map of $\Z V$-graded $R_E$-rings with involution
\[\psi:U(\mathfrak{sp}_{2n})\to \mathcal{A}(\widetilde{C_n})\] 
given by
\[ 
\begin{aligned}
\psi(e_i)&=X_i \\ 
\psi(e_n)&=\frac{\sqrt{-1}}{2} X_n
\end{aligned}
\qquad 
\begin{aligned}
\psi(f_i)&=Y_i\\ 
\psi(f_n)&=\frac{\sqrt{-1}}{2} Y_n
\end{aligned}
\qquad 
\begin{aligned}
\psi(h_i)&=u_i-u_{i+1} \\ 
\psi(h_n)&=u_n-\frac{1}{2}
\end{aligned}
\] 
for $i\in\iv{1}{n-1}$. Moreover, we have the following commutative diagram:
\begin{equation}
\begin{tikzcd}[ampersand replacement=\&,
               column sep=small]
U(\mathfrak{sp}_{2n})
 \arrow[twoheadrightarrow]{rrd}{\pi}
 \arrow[twoheadrightarrow]{d}[swap]{\psi} 
 \& \&  \\ 
\mathcal{A}(\widetilde{C_n})
 \arrow{rr}{\simeq}[swap]{\varphi_{\widetilde{C_n}}}
 \& \& A_n(\K)^{\ep} 
\end{tikzcd}
\end{equation}
where $\varphi_{\widetilde{C_n}}$ is the canonical representation by differential operators from Section \ref{sec:homomorphism-varphi}.
\end{prp}
\begin{proof}
$\widetilde{C_n}$ is connected and not in equilibrium since it has a leaf. Its underlying undirected graph has no cycles. By Theorem \ref{thm:faithfulness} and \ref{thm:local-surjectivity}, $\varphi_\Ga$ is faithful and locally surjective. Thus $\varphi_\Ga$ is an isomorphism $\mathcal{A}(\widetilde{C_n})\simeq \bigoplus_{d\in \Z V} A_n(\K)_{\ga(d)}$.
One checks that $\ga(\Z V)=\{d=\sum_v d_v v\in\Z V\mid \sum_i d_i\in 2\Z\}$. 
Thus 
$\varphi_{\widetilde{C_n}}:\mathcal{A}(\widetilde{C_n})\simeq A_n(\K)^\ep$. It remains to be noted that $\psi$ coincides with $\varphi_{\widetilde{C_n}}^{-1}\circ \pi$ on the generators $e_i,f_i$ of $U(\mathfrak{sp}_{2n})$.
\end{proof}

A weight module $M=\bigoplus_{\mu\in\mathfrak{h}^\ast}M_\mu$ over a semisimple finite-dimensional Lie algebra $\mathfrak{g}$ with Cartan subalgebra $\mathfrak{h}$ is called \emph{completely pointed} if $\dim_\K M_{\mu}\le 1$ for all $\mu\in\mathfrak{h}^\ast$.

\begin{thm}\label{thm:sp2n-quotient}
Assume $\K$ is algebraically closed. Let $M$ be an infinite-dimensional simple completely pointed $\mathfrak{sp}_{2n}$-module. 
Let $J=\Ann_{U(\mathfrak{sp}_{2n})} M$. Then $J=\ker \psi$, and hence the primitive quotient $U(\mathfrak{sp}_{2n})/J$ is isomorphic to $\mathcal{A}(\widetilde{C_n})$ as $\Z V$-graded $R_E$-rings with involution.
\end{thm}
\begin{proof}
Let $\alpha:U(\mathfrak{sp}_{2n})\to\End_\K(M)$ be the corresponding representation of the enveloping algebra. Thus $J=\ker \alpha$. By \cite{BenBriLem1997}, $M$ can be realized as an $A_n(\K)^\ep$-invariant subspace of an irreducible weight $A_n(\K)$-module. 
The argument with commuting roots in the proof of Theorem \ref{thm:sln-quotient} can be used here as well for proving $J=\operatorname{Ann}_{U(\mathfrak{sp}_{2n})}M$. Thus $M$ is faithful as a module over $A_n(\K)^\ep$ and we have the following commutative diagram:
\begin{equation}
\begin{tikzcd}[ampersand replacement=\&,
               column sep=small]
U(\mathfrak{sp}_{2n})
 \arrow{rr}{\alpha}
 \arrow[twoheadrightarrow]{rrd}{\pi}
 \arrow[twoheadrightarrow]{d}[swap]{\psi} 
 \& \& \End_\K(M) \\ 
\mathcal{A}(\widetilde{C_n})
 \arrow{rr}{\simeq}[swap]{\varphi}
 \& \& A_n(\K)^{\ep} \ar[hookrightarrow]{u}
\end{tikzcd}
\end{equation}
It follows from this diagram that $J=\ker \pi=\ker \psi$,
which proves the claim.
\end{proof}

\subsection{Primitive quotients of enveloping algebras as TGW algebras}
\label{sec:primitivequotients}

Let $\mathfrak{g}$ be a finite-dimensional simple Lie algebra. Motivated by the maps $\psi$ from Sections \ref{sec:symplectic-lie-algebras} and \ref{sec:special-linear-lie-algebras}, we consider the following question: When does there exist a surjective homomorphism from the universal enveloping algebra $U(\mathfrak{g})$ to a TGW algebra $\TGWA{\mu}{R}{\si}{t}$?
After generalizing the notion of a TGW algebra to allow $\si_i\si_j\neq\si_j\si_i$, we give a complete answer to the question in the case when $\psi(e_i)=X_i$, $\psi(f_i)=Y_i$ and $\ker \psi$ is a primitive ideal of $U(\mathfrak{g})$.

Let $(R,\si,t)$ be a TGW datum with index set $I$, except we do not require that $\si_i$ and $\si_j$ commute. We will let $G$ denote the group generated by $\si_i$. The definition of the TGW algebra $\TGWA{\mu}{R}{\si}{t}$ goes through without any problems. When some $t_i$'s are zero-divisors, there might exist a different choice of automorphisms $\si_i'$ which do commute and give the same algebra.
$(R,\si,t)$, and by abuse of language $\TGWA{\mu}{R}{\si}{t}$, will be called \emph{abelian} or \emph{non-abelian} depending on whether $G$ is abelian or not.
The following lemma shows that regular and  consistent TGW algebras must be abelian.
\begin{lem} Let $D=(R,\si,t)$ be a not necessarily abelian TGW datum. If $D$ is regular and $\mu$-consistent for some $\mu$, then $D$ is abelian.
\end{lem}
\begin{proof} Straightforward to check using the relations \eqref{eq:tgwarels}.
\end{proof}

We now prove the main theorem of this section. Recall that a module over a finite-dimensional simple Lie algebra is \emph{completely pointed} if it is a multiplicity free weight module.

\begin{thm}\label{thm:primquotient} Let $\mathfrak g$ be a finite-dimensional simple Lie algebra with
Serre generators $e_i,f_i$, $i=1,\dots,n$ and $J$ be a primitive ideal
of $U(\mathfrak g)$. The following conditions are equivalent:
\begin{enumerate}[{\rm (a)}]
\item There exists a not necessarily abelian TGW algebra $\TGWA{\mu}{R}{\si}{t}$ and a surjective homomorphism:  $\psi: U(\mathfrak g)\to\TGWA{\mu}{R}{\si}{t}$ with kernel $J$ such  
that $\psi(e_i)=X_i$, $\psi(f_i)=Y_i$;
\item There exists a simple completely pointed $\mathfrak g$-module $M$ such that
$\operatorname{Ann}_{U(\mathfrak g)} M=J$.
\end{enumerate}
\end{thm}

\begin{proof} First, let us prove that (a) implies (b). Let $Q$ denote
  the root lattice of $\mathfrak g$. The
  adjoint action of the Cartan subalgebra $\mathfrak h$ defines
a  $Q$-grading of $U(\mathfrak g)$ and it induces a $Q$-grading on $\TGWA{\mu}{R}{\si}{t}$.
Thus, $\psi$ is a homomorphism of graded algebras and therefore
$\psi(U(\mathfrak g)_0)=R_0$. Note that $R_0$ is generated by
$X_iY_i$ and $\psi(\mathfrak h)$. Hence $R_0$ is commutative.
By Duflo's theorem there exists a simple highest weight $\mathfrak g$-module $M$
whose annihilator is $J$. By Proposition 9.6.1 in [Dix] any weight
space $M_\gamma$ is a simple $R_0$-module. Therefore
$\operatorname{dim} M_\gamma=1$, i.e. $M$ is completely pointed.  

Now let us prove that (b) implies (a). 
Let us assume first that $M$ is infinite-dimensional.
By the classification of completely pointed weight modules \cite{BenBriLem1997},  $\mathfrak g=\mathfrak{sl}_n$ or
$\mathfrak{sp}_{2n}$. Then (a) follows from
Theorem \ref{thm:sln-quotient} and Theorem \ref{thm:sp2n-quotient}.

Thus we may assume that $M$ is finite dimensional. 
Let $\Gamma$ denote the set of weights of $M$ and $\mathcal A=\operatorname{End}(M)$.
The homomorphism $\rho:U(\mathfrak g)\to \mathcal A$ is surjective by
the Jacobson density theorem.
Note that $\mathcal A$ has a $Q$-grading $\mathcal A=\bigoplus_{\alpha\in Q} \mathcal A_\alpha$ defined
by 
$$\mathcal A_\alpha=\{X\in\mathcal A | X M_\gamma\subset M_{\gamma+\alpha}\forall \gamma\in\Gamma\}.$$ 
We set $R=\mathcal A_0$, $X_i=\rho(e_i), Y_i=\rho(f_i)$, $t_i=Y_iX_i$ and claim that
$\mathcal A$ is a TGW algebra with all $\mu_{ij}=1$. Since $\mathcal A$
is simple, it
suffices to prove the existence of automorphisms $\sigma_i$ of $R$
such that  $X_i r=\sigma_i(r) X_i$ and $Y_i r=\sigma_i^{-1}(r) Y_i$
for all $r\in R$.
Let $E_\gamma$ denote the $\mathfrak h$-invariant projector on
$M_\gamma$. Clearly $E_\gamma$ for all $\gamma\in\Gamma$ is a basis of $R$. 
Let $\alpha_i$ denote the simple root (weight of $e_i$).
If $\gamma-\alpha_i\in \Gamma$ we set $\sigma_i(E_\gamma)=E_{\gamma-\alpha_i}$.
If $\gamma-\alpha_i\notin \Gamma$, we set  $\sigma_i(E_\gamma)=E_{\gamma+k_\gamma\alpha_i}$
where $k_\gamma$ is the maximal integer $k$ such that
$\gamma+k\alpha_i\in \Gamma$. In the former case $E_\gamma X_i=X_iE_{\gamma-\alpha_i}$
since $X_iM_{\gamma-\alpha_i}=M_\gamma$. In the latter case $E_\gamma X_i=0$ since 
$M_\gamma\notin X_iM$ and by construction
$X_iE_{\gamma+k_\gamma\alpha_i}=0$. The proof of $Y_ir=\sigma_i^{-1}(r) Y_i$ is similar. That concludes the proof.
\end{proof}

\begin{rem} \label{rem:abelian}
If $M$ is infinite-dimensional then $\sigma_i$ commute by
  our construction of the subalgebra in a Weyl algebra in Theorem \ref{thm:primquotient}.
In the case of finite-dimensional $M$, the choice of
  $\sigma_i$ is not unique. The choice used in our proof usually does not imply that the group $G$ generated by $\sigma_i$ is abelian. For example, if $M$ is a minuscule representation the choice of
  $\sigma_i$ as in the proof of Theorem \ref{thm:primquotient} gives
  $G$ isomorphic to the Weyl group of $\mathfrak g$.
\end{rem}

\subsection{Completely pointed loop modules} \label{sec:affine}
Let $\hat{\mathfrak g}$ denote the affinizaion of $\mathfrak g$. Recall that  
 $$\hat{\mathfrak g}=\mathfrak g\otimes\K[t,t^{-1}]\oplus \K c\oplus\K t\frac{\partial}{\partial t},$$ 
where $c$ is the central element. For any $\mathfrak g$-module $M$ let
$\hat M=M\otimes \K[t,t^{-1}]$ denote the corresponding loop module
(with trivial action of $c$). Recall that Serre generators of
$\hat{\mathfrak g}$ are $e_i\otimes 1,i=1,\dots n,e_0\otimes t$ and  $f_i\otimes 1,i=1,\dots,n,f_0\otimes t^{-1}$,
where is $e_0$ is the lowest and $f_0$ is the higest vector in the
adjoint representation of $\mathfrak g$.

\begin{thm}\label{affine} Let $M$ be a simple completely pointed
  $\mathfrak g$-module, $J=\operatorname{Ann}_{U(\hat{\mathfrak g})}\hat M$.
 There exists a TGW algebra $\TGWA{\mu}{R}{\si}{t}$ and a surjective homomorphism:  $\psi: U(\hat{\mathfrak g})\to\TGWA{\mu}{R}{\si}{t}$ with kernel $J$ such  
then $\psi(e_i\otimes 1)=X_i$, $\psi(f_i\otimes 1)=Y_i$ for
$i=1,\dots n$ and $\psi(e_0\otimes t)=X_0$, $\psi(f_0\otimes t^{-1})=Y_0$.
\end{thm}
\begin{proof} Let $\psi: U(\hat{\mathfrak  g})\to\operatorname{End}(\hat M)$
be the natural homomorphism. We claim that $\mathcal A=\operatorname{Im}\psi$ is a TGW
algebra with $R$ generated by $X_iY_i$, $i=0,\dots,n$ and
$t\frac{\partial}{\partial t}$. It is easy to see that
$R$ is isomorphic to $\K[t\frac{\partial}{\partial  t}]\otimes \psi (U(\mathfrak g))_0$. 
For $i=1,\dots,n$ we define $\sigma_i$ on
$\psi(U(\mathfrak g)_0)$ as in Theorem  \ref{thm:primquotient} and
set $\sigma_0=\sigma_1^{-a_1}\cdots\sigma_n^{-a_n}$
where $a_1,\dots,a_n$ are the coefficients in the decomposition of the
highest root of $\mathfrak g$ into linear combination of simple roots.
We set $\sigma_i(t\frac{\partial}{\partial t})=t\frac{\partial}{\partial t}$ for $i=1,\dots,n$, 
and $\sigma_0(t\frac{\partial}{\partial t})=t\frac{\partial}{\partial t}+1$.
All relations are clear from the construction and it remains to show
that  $\mathcal A$ does not have non-zero ideals with zero
intersection with $R$. 

We use the $\mathbb Z$-grading of  $\mathcal A$ induced 
by the adjoint action of $t\frac{\partial}{\partial t}$. By our
construction $\mathcal A_0=\K[t\frac{\partial}{\partial  t}]\otimes \psi (U(\mathfrak g))$.
Every ideal $I$ in $\mathcal A$ is homogeneous.  Since
$R\subset \mathcal A_0$ we have $I\cap R=I_0\cap R$ for any ideal
$I\subset\mathcal A$. Let $I$ be  ideal of $\mathcal A$ such that
$I\cap R=0$. Then $I_0\cap R=0$. Hence $I_0=0$ as follows easily from  Theorem \ref{thm:primquotient}. 
Suppose $I_0=0$ but $I_k\neq 0$ for some $k$. There exists $h\in\mathfrak h$ such that 
$\psi(h\otimes t^{-k})$ is invertible in $\operatorname{End}(M)$. Then $\psi(h\otimes t^{-k})I_k\neq 0$,
which is a contradiction.
\end{proof}

\section{Appendix: Relation to a previous family of TGW algebras}
\label{sec:symmetric}
The following TGW algebras (along with quantum analogues) were first defined in \cite{Hartwig2010}. We show that modulo a graded ideal, they are special cases of $\mathcal{A}(\Ga)$.

Let $n$ be a positive integer and $C=(a_{ij})_{i,j=1}^n$ be an $n\times n$ symmetric generalized Cartan matrix (GCM). Define a TGW datum $(R_C,\si_C,t_C)$ as follows:
\begin{gather}
R_C=\K[H_{ij}^{(k)}\mid 1\le i<j\le n,\; k=a_{ij},a_{ij}+2,\ldots,|a_{ij}|] \\
\si_C=(\si_i)_{i=1}^n,\quad \si_r(H_{ij}^{(k)})=
\begin{cases}
H_{ij}^{(k)}+H_{ij}^{(k-2)},& \text{$r=j$ and $k>a_{ij}$,}\\
H_{ij}^{(k)}-H_{ij}^{(k-2)}+H_{ij}^{(k-4)}-\cdots \pm H_{ij}^{(a_{ij})}, & r=i,\\ 
H_{ij}^{(k)},& \text{otherwise,}
\end{cases} \\ 
t_C=(t_i)_{i=1}^n,\quad
t_i=H_{i1}H_{i2}\cdots H_{in},\quad
H_{ij}=
\begin{cases}
H_{ij}^{(|a_{ij}|)}, & i<j, \\ 
1,&i=j,\\
\si_i^{-1}(H_{ji}^{(|a_{ij}|)}), & i>j.
\end{cases}
\end{gather}
Define $\mathcal{T}(C)=\mathcal{A}(R_C,\si_C,t_C)$.
The main point of these algebras is that the GCM associated to $(R_C,\si_C,t_C)$ is exactly $C$ (see \cite{Hartwig2010} for details).

We now define a multiquiver $\Ga_C=(V,E,s,t)$ associated to $C$. Let 
\begin{gather}
V=\{1,2,\ldots,n\},\qquad  E=\big\{(i,j)\in V\times V\mid i<j, a_{ij}\neq 0\big\},\\ 
s((i,j))=(i,|a_{ij}|),\quad t((i,j))=(j,|a_{ij}|)\quad \forall (i,j)\in E.
\end{gather}
Note that $\Ga_C$ is a symmetric simple quiver. Conversely, one can show that any finite symmetric simple quiver satisfying those two conditions is sign-equivalent to $\Ga_C$ for some symmetric GCM $C$. In what follows we identify the group $\Z V$ with $\Z^n$ in the obvious way.

\begin{prp}
Let $C$ be a symmetric GCM and $\Ga_C$ the corresponding multiquiver as defined above.
Then there is a surjective homomorphism of $\Z^n$-graded $\K$-algebras
\[\mathcal{F}:\mathcal{T}(C)\to \mathcal{A}(\Ga_C).\] 
satisfying $\mathcal{F}(X_i)=X_i$ and $\mathcal{F}(Y_i)=Y_i$ for all $i\in V$.
\end{prp}
\begin{proof}
Put $\Ga=\Ga_C$. Let $F:R_C\to R_E$ be the unique $\K$-algebra homomorphism determined by
\begin{equation}
F(H_{ij}^{(|a_{ij}|-2k)})=(\si_j-\Id)^k(u_{(i,j),i}),\qquad\forall i<j,\; k=0,1,\ldots,|a_{ij}|.
\end{equation}
By definition, $F\circ (\si_j-\Id) = (\si_j^\Ga-\Id)\circ F$ on $H_{ij}^{(k)}$. So $F\circ \si_j=\si_j^\Ga\circ F$ on $H_{ij}^{(k)}$. Then $F\circ \si_i = \si_i^\Ga\circ F$ on $H_{ij}^{(k)}$ too since $\si_i\si_j=\Id$ on $H_{ij}^{(k)}$. Also, $F\circ\si_r=\si_r^\Ga\circ F$ on $H_{ij}^{(k)}$ for $r\notin\{i,j\}$ because $\si_r(u_{(i,j)})=u_{(i,j)}$ for $r\notin\{i,j\}$. This shows that $F\circ\si_i = \si_i^\Ga\circ F$ on $R_C$. Furthermore, by a direct calculation using that $\si_j(u_{(i,j)},j)=u_{(i,j),i}$ one verifies that $F(t_i)=t_i^\Ga$ for all $i$. This means that $F$ is a morphism of TGW data, as defined in \cite{FutHar2012b}, and that applying the functor $\mathcal{A}$ from \cite{FutHar2012b} we obtain a homomorphism
\[\mathcal{F}:\mathcal{A}(R_C,\si_C,t_C)\to\mathcal{A}(\Ga)\]
of $\Z^n$-graded algebras.
One checks that $F(H_{ij}^{(2-|a_{ij}|)})$ is a polynomial in $u_{(i,j)}$ of degree $1$ for any $(i,j)\in E$. This proves that $F$ is surjective. By \cite[Cor.~3.2]{FutHar2012b}, $\mathcal{F}$ is also surjective.
\end{proof}

\bibliographystyle{siam} 

\end{document}